\documentclass[12pt]{amsart}
\usepackage{enumerate,amssymb,amsfonts,latexsym,psfrag}
\usepackage{amscd,amsmath}
\usepackage[dvips]{graphicx,epsfig}
\usepackage{epstopdf}
\usepackage[all]{xy}
\usepackage{fullpage}
\usepackage{subfig}
\usepackage{multirow,array}
\usepackage{color}
\usepackage{mathtools}
\usepackage{amsthm}

\usepackage{url}

\newtheorem{thm}{Theorem}[section]

\newtheorem{cor}[thm]{Corollary}
\newtheorem{lem}[thm]{Lemma}
\newtheorem{prop}[thm]{Proposition}

\theoremstyle{remark}

\theoremstyle{definition}
\newtheorem{defn}{Definition}[section]
\newtheorem{rem}{Remark}[section]

\numberwithin{equation}{section}
\numberwithin{figure}{section}

\font\nt=cmr7

\def\note#1
{\marginpar
{\nt $\leftarrow$
\par
\hfuzz=20pt \hbadness=9000 \hyphenpenalty=-100 \exhyphenpenalty=-100
\pretolerance=-1 \tolerance=9999 \doublehyphendemerits=-100000
\finalhyphendemerits=-100000 \baselineskip=6pt
#1}\hfuzz=1pt}

\newcommand{\di}{\partial}

\newcommand{\ra}{\rightarrow}

\def\ssk{\smallskip}
\def\msk{\medskip}
\def\bsk{\bigskip}

\def\nin{\noindent}

\newcommand{\diam}{\operatorname{diam}}
\newcommand{\dist}{\operatorname{dist}}

\renewcommand{\mod}{\operatorname{mod}}

\newcommand{\id}{\operatorname{id}}

\newcommand{\Per}{\operatorname{Per}}

\newcommand{\Orb}{\operatorname{Orb}}

\newcommand{\tip}{\operatorname{tip}}

\newcommand{\eps}{{\varepsilon}}

\newcommand{\de}{{\delta}}

\newcommand{\La}{{\Lambda}}
\newcommand{\si}{{\sigma}}

\newcommand{\AAA}{{\mathcal A}}

\newcommand{\II}{{\mathcal I}}

\newcommand{\OO}{{\mathcal O}}

\newcommand{\N}{{\mathbb N}}

\newcommand{\R}{{\mathbb R}}

\def\BPhi{{\boldsymbol{\BPhi}}}
\def\B0{{\mathbf{0}}}

\newcommand{\Jac}{\operatorname{Jac}}

\newcommand{\Dom}{\operatorname{Dom}}

\catcode`\@=12

\def\Empty{}
\newcommand\oplabel[1]{
  \def\OpArg{#1} \ifx \OpArg\Empty {} \else
  	\label{#1}
  \fi}
		
\newcommand{\comm}[1]{}
\newcommand{\comment}[1]{}

\usepackage[titletoc,title,header]{appendix}

\usepackage{titlesec}

\titleformat{\section}
  {\Large\bfseries\center}{\thesection.}{0.5em}{}
  \titlespacing*{\section} {0pt}{1em}{2.3ex plus .2ex}
\titlelabel{\thetitle.\quad}
\titleformat{\subsection}
  [runin]{\normalfont\bfseries}{\thesubsection.}{0.5em}{}[.]
\titlespacing*{\subsection}{0pt}{1em}{2.3ex plus .2ex}

\usepackage{titletoc}
\titlecontents{section}[1.5em]
  {\normalfont\bfseries}
  {\contentslabel{1.5em}}
  {\hspace*{-2.3em}}
 {\titlerule*[1pc]{}\contentspage}

\allowdisplaybreaks
\newcommand\numberthis{\addtocounter{equation}{1}\tag{\theequation}}

\begin{document}

\bigskip\bigskip

\title[H\'enon renormalization]{Renormalization of $ C^r $ H\'enon map $ \colon $ Two dimensional embedded map in three dimension }

\address {Hong-ik University }
\date{December 29, 2014}
\author{Young Woo Nam}
\thanks{College of Science and Technology, Hongik University at Sejong, Korea. \newline 
Email $ \colon $ namyoungwoo\,@\,hongik.ac.kr}

\begin{abstract}
We study renormalization of highly dissipative analytic three dimensional H\'enon maps
$$ F(x,y,z) = (f(x) - \eps(x,y,z),\ x,\ \de(x,y,z)) $$
where $ \eps(x,y,z) $ is a sufficiently small perturbation of $ \eps_{2d}(x,y) $. Under certain conditions, $ C^r $ {\em single invariant surfaces} each of which is tangent to the invariant plane field over the critical Cantor set exist for $ 2 \leq r < \infty $. The $ C^r $ conjugation from an invariant surface to the $ xy- $plane defines renormalization two dimensional $ C^r $ H\'enon-like map. It also defines two dimensional {\em embedded} $ C^r ${\em H\'enon-like maps} in {three dimension}. In this class, universality theorem is re-constructed by conjugation. Geometric properties on the critical Cantor set in invariant surfaces are the same as those of two dimensional maps --- {non existence} of {the continuous line field}, and {unbounded geometry}. The set of embedded two dimensional H\'enon-like maps is open and dense subset of the parameter space of average Jacobian, $ b_{F_{2d}} $ 
for any given smoothness, $ 2 \leq r < \infty $. 
\end{abstract}

\maketitle


\thispagestyle{empty}

\setcounter{tocdepth}{1}

\tableofcontents


\newpage
\renewcommand{\labelenumi} {\rm {(}\arabic{enumi}{)}}

\section{Introduction}
Renormalization is for the one dimensional maps for a few recent decades by many authors in various papers. Some of main results and historical facts of renormalization theory of one dimensional maps are in \cite{dFdMP} and references therein. Renormalization of higher dimensional maps was started by Coullet, Eckmann and Koch in \cite{CEK}. 
Period doubling renormalization of analytic H\'enon map with strong dissipativeness was introduced in \cite{CLM}. The average Jacobian $ b_F $ of infinitely renormalizable H\'enon-like map, $ F $, is defined 
$$ b_F = \exp \int_{\OO_F} \log \Jac F \,d\mu $$
where $ \OO_F $ is the critical Cantor set and $ \mu $ is the ergodic measure on $ \OO_F $. 
Carvalho, Lyubich and Martens in \cite{CLM} proved Universality Theorem and showed geometric properties of the critical Cantor set which are different from those of one dimensional maps. For instance, generic unbounded geometry of the critical Cantor set in the parameter space of the average Jacobian was shown and this geometric property is generalized for the full Lebesgue measure set in \cite{HLM}. 
\ssk \\
\nin H\'enon renormalization is generalized for three dimensional analytic H\'enon-like map in \cite{Nam1}. For instance, the universal asymptotic expression of $ R^nF $ is
$$ \Jac R^nF(x,y,z) = b_F^{2^n}a(x)(1 +O(\rho^n)) $$ 
where $ a(x) $ is analytic and positive for $ 0 < \rho < 1 $. However, the universal expression of Jacobian determinant of three dimensional renormalized map does not imply the Universal Theorem because the Jacobian determinant, $ \Jac R^nF = \di_y \eps_n \di_z \de_n - \di_z \eps_n  \di_y \de_n $ contains partial derivatives of both $ \eps $ and $ \de $. Moreover, infinitely renormalizable H\'enon map has maximal Lyapunov exponent is zero. Thus $ \ln b $ is the other exponent for two dimensional map. However, since $ \ln b_F $ for three dimensional map is not an exponent but the sum of Lyapunov exponents. Thus two universal numbers for three dimensional maps would be required in order to explain geometric properties of  $ \OO_F $. One of the universal numbers is a counterpart of the average Jacobian of two dimensional map. The universal numbers, $ b_1 $ and $ b_2 $ which represent two dimensional H\'enon-like map in three dimension and contraction from the third dimension were found in \cite{Nam1} under certain conditions. For the precise formulation, see \S \ref{subsec-toy model Henon map}. 
\ssk \\
In the present paper, three dimensional H\'enon-like maps with certain conditions has single invariant $ C^r $ surfaces for any natural number $ 2 \leq r < \infty $ and it is asymptotically slanted plane (Proposition \ref{invariant surfaces on each deep level}). The map from invariant surface to $ xy- $plane defines the renormalization of $ C^r $ H\'enon-like maps and it is the same as the analytic definition of H\'enon renormalization (Proposition \ref{2d scaling map of Cr conjugation})
$$ RF = \La \circ H \circ F^2 \circ H^{-1} \circ \La^{-1} . $$ 
Moreover, Universality Theorem for $ C^r $ H\'enon-like map is re-constructed by invariant surfaces (Theorem \ref{Universality of Cr Henon maps}). It defines the {\em embedded two dimensional H\'enon-like map} in {\em three dimension}. Moreover, two dimensional $ C^r $ H\'enon-like map is embedded in three dimension generically in the set of parameter space of average Jacobian (Theorem \ref{thm-open dense subset of Cr Henon-like maps}). The universal numbers of three dimensional H\'enon-like map, $ b_1 $ which is the average Jacobian of two dimensional $ C^r $ H\'enon-like map and $ b_2 \equiv b_F/b_1 $, we would show the unbounded geometry of $ \OO_F $ for almost everywhere in the parameter space of $ b_1 $ of embedded $ C^r $ H\'enon-like maps (Theorem \ref{Unbounded geometry for model maps}).


\bsk

\section{Preliminaries} \label{preliminaries}  
\subsection{Notations}
For the given map $ F $, if a set $ A $ is related to $F$, then we denote it to be $ A(F) $ or $ A_F $ and $ F $ can be skipped if there is no confusion without $ F $. The domain of $ F $ is denoted to be $ \Dom(F) $. 
If $ F(B) \subset B $, then we call $ B $ is an (forward) invariant set under $ F $. 
The set $ \overline{A} $ in the given topology is called the closure of $ A $. 
For three dimensional map, let us the projection from $ \R^3 $ to its $ x- $axis, $ y- $axis and $ z- $axis be $ \pi_x $, $ \pi_y $ and $ \pi_z $ respectively. Moreover, the projection from $ \R^3 $ to $ xy- $plane be $ \pi_{xy} $ and so on.
\ssk \\
Let $ C^r(X) $ be the Banach space of all real functions on $ X $ for which the $ r^{th} $ derivative is continuous. The $ C^r $ norm of $ h \in C^r(X) $ is defined as follows
$$ \| h \|_{C^r} = \max_{ 1 \leq \,k \leq \,r } \left\{ \| h \|_0,\ \| D^kh \|_0 \right\} . $$
For analytic maps, since $ C^0 $ norm bounds $ C^r $ norm for any $ r \in \N $, we often use the norm, $ \| \cdot \| $ instead of $ \| \cdot \|_0 $ or $ \| \cdot \|_{C^k} $. For the two sets $ S $ and $ T $ in $ \R^3 $, the minimal distance of two sets is defined as
$$ \dist_{\min}(S, T) = \inf \;\{ \dist (p, q)\; | \; p \in S \ \text{and} \ q \in T \}
$$
The set of periodic points of the map $ F $ is denoted by $ \Per_F $.  
$ A = O(B) $ means that there exists a positive number $ C $ such that $ A \leq CB $. Moreover, $ A \asymp B $ means that there exists a positive number $ C $ which satisfies $ \dfrac{1}{C} B \leq A \leq CB$.
\msk

\subsection{Renormalization of two and three dimensional H\'enon-like maps} \label{2d renorm operator}
Two dimensional H\'enon-like map is defined as 
$$ F(x,y ) = (f(x) -\eps(x,y),\ x)   $$
where $ f $ is a unimodal map. Assume that the norm of $ \eps $ is sufficiently small and $ F $ is orientation preserving map. 
Since $ F^2 $ is not H\'enon-like map, the non linear scaling map for renormalization of H\'enon-like map, $ F $. The horizontal map of $ F $ is defined 
$$ H(x,y) = (f(x) -\eps(x,y),\ y). $$
The {\em period doubling} renormalization of $ F $ is defined as 
$$ RF = \La \circ H \circ F^2 \circ H^{-1} \circ \La^{-1} $$
where $ \La(x,y) = (sx,\ sy) $ for the appropriate number $ s<-1 $ in \cite{CLM}. 
Moreover, H\'enon renormalization theory is extended for three dimensional H\'enon-like map in \cite{Nam1} with third coordinate map as follows 
$$ F(x,y,z) = (f(x) - \eps(x,y,z),\ x,\ \de(x,y,z)). $$
We assume that the norms of both $ \eps $ and $ \de $ are sufficiently small and that the three dimensional map $ F $ is analytic throughout this paper. The domain of $ F $ is cubic box and $ F $ has two fixed points and {\em sectionally dissipative} at these points. The horizontal-like map is defined
$$ H(x,y,z) = (f(x) -\eps(x,y,z),\ y,\ z- \de(y,f^{-1}(y),0)). $$
Thus the (period doubling) renormalization of three dimensional map is the natural extension of two dimensional H\'enon-like map as follows 
$$ RF = \La \circ H \circ F^2 \circ H^{-1} \circ \La^{-1} $$
where $ \La(x,y,z) = (sx,\ sy,\ sz) $ for the appropriate number $ s<-1 $. 

\msk
\subsection{Basic facts}
Let the set of infinitely renormalizable maps be $ \II(\bar \eps) $ where the norm $ \| \eps \| $ and $ \| \de \| $ (for three dimensional maps) are bounded above by $ O(\eps) $ where $ \bar\eps $ is a small enough positive number. The following definitions and facts are common in both two and three dimensional H\'enon-like maps in $ \II(\bar \eps) $. \ssk \\
%
%
If $ F $ is $ n- $times renormalizable, then $ R^kF $ is defined as the renormalization of $ R^{k-1}F $ for $ 2 \leq k \leq n $. Denote $ \Dom(F) $ to be the 
box region, $ B $. If the set $ B $ is emphasized with the relation of a certain map $ R^kF $, for example, then denote this region to be $ B(R^kF) $. 
\ssk \\
$ F_k $ denotes $ R^kF $ for each $ k $. Let the coordinate change map which conjugates $ F_k^2|_{\La_k^{-1}(B)} $ and $ RF_k $ is denoted by 
\begin{equation*}
\begin{aligned}
\psi^{k+1}_v \equiv H_k^{-1} \circ \La_k^{-1} \colon \Dom(RF_k) \ra \La_k^{-1}(B)
\end{aligned} 
\end{equation*}
\nin where $ H_k $ is the horizontal-like diffeomorphism and $ \La_k $ is dilation with each appropriate constants $ s_k < -1 $ for each $ k $. Denote $ F_k \circ \psi^{k+1}_v $ by $ \psi^{k+1}_c $. The word of length $ n $ in the Cartesian product, $ W^n \equiv \{ v, c \}^n $ is denoted by $ {\bf w}_n $ or simply $ {\bf w} $. 
Express the compositions of $ \psi^{j}_v $ and $ \psi^{j}_c $ for $ k \leq j \leq n $ as follows
\begin{equation*}
\begin{aligned}
\Psi^n_{k,\,{\bf w}} &= \psi^{k}_{w_1} \circ \psi^{k+1}_{w_2} \circ \cdots \circ \psi^n_{w_{n-k}} 
\end{aligned} 
\end{equation*}
where each $ w_i $ is $ v $ or $ c $ 
and the word $ {\bf w} = (w_1w_2 \ldots w_{n-k}) $ in $ W^{n-k} $. 
The map $ \Psi^n_{k,\,{\bf w}} $ is from $ B(R^nF) $ to $ B(R^kF) $. 
Denote the region $ \Psi^n_{k,\,{\bf w}}(B(R^nF)) $ by $ B^n_{k, {\bf w}} $. In particular, denote $ B^n_{0,\,{\bf w}} $ by $ B^n_{{\bf w}} $ for simplicity. 
We see that 
\begin{equation} \label{eq-diameter of box Bn}
\diam (B^n_{{\bf w}}) \leq C\si^n
\end{equation}
where $ {\bf w} $ is any word of length $ n $ in $ W^n $ for some $ C>0 $ in \cite{CLM} or \cite{Nam}. If $ F $ is a infinitely renormalizable H\'enon-like map, then it has invariant Cantor set
\begin{equation*}
\OO_F = \bigcap_{n=1}^{\infty} \bigcup_{{\bf w} \in\,W^n} B^n_{{\bf w}} 
\end{equation*}
and $ F $ acts on $ \OO_F $ as a dyadic adding machine. The counterpart of the critical value of unimodal renormalizable map is called the {\em tip}
\begin{equation*}
\{\tau_F \} \equiv \bigcap_{n\geq \;\! 0} B^n_{{\bf v}} 
\end{equation*}
where $ {\bf v} = v^n $ for every $ n \in \N $. The word $ {\bf w} \in W^{\infty} $ for each $ w \in \OO $ is called the {\em address} of $ w $. Similarly, the word with finite length $ {\bf w}_n \in W^n $ corresponding the region, $ B^n_{{\bf w}_n} $ is called the address of box. Moreover, since each box, $ B^n_{{\bf w}_n} $ contains a unique periodic point with minimal period, $ 2^n $, the {\em address} of {\em periodic point} is also defined as that of $ B^n_{{\bf w}_n} $. The first successive finite concatenation of the given address, $ \bf w $ is called the {\em subaddress} of $ w $. 
By Distortion Lemma and the average Jacobian with invariant measure, we see the following lemma.
\begin{lem} \label{Distortion lemma}
For any piece $ B^n_{{\bf w}} $ at any point $ w = (x,y,z) \in B^n_{{\bf w}} $, the Jacobian determinant of $ F^{2^n} $ is 
\begin{equation}
\Jac F^{2^n}(w) = b_F^{2^n}(1+ O(\rho^n))
\end{equation}
where $ b $ is the average Jacobian of $ F $ for some $ 0 < \rho <1 $. 
\end{lem} 
Then there exists the asymptotic expression of $ \Jac R^nF $ for the map $ F \in \II(\bar \eps) $ with $ b_F $ and the universal function.
\ssk
\begin{thm}[\cite{CLM} and \cite{Nam}] \label{Universality of the Jacobian}
For the map $ F \in \II(\bar \eps) $ with small enough positive number $ \bar \eps $, the Jacobian determinant of $ n^{th} $ renormalization of $ F $ is as follows
\begin{equation}
\Jac R^nF = b_F^{2^n}a(x)\, (1+O(\rho^n))
\end{equation}
where $ b_F $ is the average Jacobian of $ F $ and $ a(x) $ is the universal positive function for $ n \in \N $ and for some $ \rho \in (0,1) $.
\end{thm}
\ssk 
\nin Denote the tip, $ \tau_{F_n} $ to be $ \tau_n $ for $ n \in \N $. The definitions of tip and $ \Psi^n_{k,\,{\bf v}} $ imply that $ \Psi^n_{k,{\bf v}}(\tau_n) = \tau_k $ for $ k < n $. Then after composing appropriate translations, tips on each level moves to the origin as the fixed point
\begin{equation*}
\Psi^n_k(w) = \Psi^n_{k,{\bf v}}(w + \tau_n) - \tau_k
\end{equation*}
for $ k < n $. Notations with the subscript, $ {\bf v} $ is strongly related to the tip. For instance, $ B^n_{k,{\bf v}} $ contains the tip, $ \tau_k $ for every $ n > k $ and $ \Psi^n_{k, {\bf v}} $ is the map from the tip, $ \tau_n $ to the tip $ \tau_k $ for every $ n > k $. Thus in order to emphasize the tip on every deep level, we sometimes use the notation $ B^n_{k,{\tip}} $ or $ \Psi^n_{k, {\tip}} $ instead of $ B^n_{k,{\bf v}} $ or $ \Psi^n_{k, {\bf v}} $. Moreover, if we need to distinguish three dimensional notions from two dimensional one, then we use the subscript, $ 2d $. For example, $ {}_{2d}^{}\Psi^n_{k} $, $ {}_{2d}^{}B^n_{k,{\bf v}} $, $ {}_{2d}^{}t_{n,\,k} $, $ {}_{2d}^{}S^n_k(w) $ and so on.

\msk

\subsection{Three dimensional coordinate change map, $ \Psi^n_k $}
The map $ \Psi^n_k $ is separated non linear part and dilation part after reshuffling
\msk
\begin{equation} \label{eq-coordinate change map 3D}
\begin{aligned}
\Psi^n_k(w) = 
\begin{pmatrix}
1 & t_{n,\,k} & u_{n,\,k} \\[0.2em]
& 1 & \\
& d_{n,\,k} & 1
\end{pmatrix}
\begin{pmatrix}
\alpha_{n,\,k} & & \\
& \si_{n,\,k} & \\
& & \si_{n,\,k}
\end{pmatrix}
\begin{pmatrix}
x + S^n_k(w) \\
y \\[0.2em]
z + R^n_k(y)     
\end{pmatrix} 
\end{aligned} \msk
\end{equation} 
where $ \alpha_{n,\,k} = \si^{2(n-k)}(1 + O(\rho^k)) $ and $ \si_{n,\,k} = (-\si)^{n-k}(1 + O(\rho^k)) $. The non-linear map $ x + S^n_k(w) $ has following asymptotic with the universal diffeomorphism $ v_*(x) $.
\begin{lem} \label{Asymptotic non-linear part}
Let $ x + S^n_k(w) $ be the first coordinate map of three dimensional coordinate change map in \eqref{eq-coordinate change map 3D} for infinitely renormalizable H\'enon-like map. Then 
\begin{equation} 
|[x +S^n_0(x,y,z)] - [v_*(x) + a_{F,1}y^2 + a_{F,2}yz + a_{F,3}z^2]| = O(\rho^n)
\end{equation}
where constants $ |a_{F,1}| $, $ |a_{F,2}| $ and $ |a_{F,3}| $ are $ O(\bar \eps) $ for $ \rho \in 
(0,1) $. Moreover, for each fixed $ y $ and $ z $, the above asymptotic has $ C^1 $ convergence with the variable $ x $. 
\end{lem}
\nin The constants $ t_{n,\,k} $, $ u_{n,\,k} $ and $ d_{n,\,k} $ converges to some numbers --- say $ t_{*,\,k} $, $ u_{*,\,k} $ and $ d_{*,\,k} $ respectively --- super exponentially fast as $ n \ra \infty $. Moreover, estimation of the above constants is following
\begin{equation} \label{eq-estimation of t,u,d}
|t_{n,\,k}|,\ |u_{n,\,k}|,\ |d_{n,\,k}| \leq C\bar\eps^{2^k}
\end{equation} 
for $ k<n $ and for some constant $ C>0 $. Lemma 5.1 in \cite{Nam2} contains the detailed calculation for these constants. Moreover, Lemma 5.2 in \cite{Nam2} implies that 
\begin{equation} \label{eq-C1 convergence of R-n-k}
\| R^n_k \|_{C^1} \leq C \si^n
\end{equation}
for some $ C>0 $ independent of $ n $. 
\comm{*********************************
In this paper, we confuse the map $ \Psi^n_{k,\,{\bf v}} $ with $ \Psi^n_k $ to obtain the simpler expression of each coordinate map of $ \Psi^n_{k,\,{\bf v}} $. For example, the third coordinate expression of $ \Psi^n_k $
\begin{equation*}
\si_{n,\,k}\,d_{n,\,k}\,y + \si_{n,\,k}\,\big[\,z + R^n_k(y)\,\big] 
\end{equation*}
means that $ \si_{n,\,k}\,d_{n,\,k}\,(y + \tau_n^y) + \si_{n,\,k}\,\big[\,z + \tau_n^z + R^n_k(y + \tau_n^y)\,\big] - \tau_k $ where $ \tau_n = (\tau_n^x,\; \tau_n^y,\; \tau_n^z) $. By the same way, the first coordinate map
\begin{equation*}
\alpha_{n,\,k}\,\big[\,x + S^n_k(w)\,\big] + \si_{n,\,k}\,t_{n,\,k}\,y + \si_{n,\,k}\,u_{n,\,k}\,\big[\,z + R^n_k(y) \,\big]
\end{equation*}
means that
\msk
\begin{equation*}
\alpha_{n,\,k}\,\big[\,(x + \tau_n^x) + S^n_k(w + \tau_n)\,\big] + \si_{n,\,k}\,t_{n,\,k}\,(y + \tau_n^y) + \si_{n,\,k}\,u_{n,\,k}\,\big[\,(z + \tau_n^z) + R^n_k(y+ \tau_n^y) \,\big] - \tau_k
\end{equation*} \ssk
for $ k < n $.
*******************************}
\nin Recall the following definitions 
for later use
\ssk
\begin{equation*}
\begin{aligned}
\La_n^{-1}(w) =  \si_n \cdot w , \quad  \psi^{n+1}_v(w) =  H^{-1}_n (\si_n w) , \quad   
\psi^{n+1}_c(w) =  F_n \circ H^{-1}_n (\si_n w)  \\[0.3em]
\psi^{n+1}_v (B(R^{n+1}F)) = B^{n+1}_v ,  \qquad \psi^{n+1}_c (B(R^{n+1}F)) = B^{n+1}_c
\end{aligned} \ssk
\end{equation*}
for each $ n \in \N $.

\msk
\subsection{Toy model H\'enon-like maps} \label{subsec-toy model Henon map}
Let H\'enon-like map satisfying $ \eps(w) = \eps(x,y) $, that is, $ \di_{z} \eps \equiv 0 $ be {\em toy model H\'enon-like map}. Denote the toy model map by $ F_{\mod} $. Then the projected map $ \pi_{xy} \circ F_{\mod} = F_{2d} $ from $ B $ to $ \R^2 $ is exactly two dimensional H\'enon-like map. 
If $ F_{\mod} $ is renormalizable, then we have $ \pi_{xy} \circ RF_{\mod} = RF_{2d} $.

\begin{prop} \label{renormalization of toy model}
Let $ F_{\mod} = (f(x) - \eps_{2d}(x,y),\ x,\ \de(w)) $ be a toy model diffeormorphism in $ \II(\bar \eps) $. Then $ n^{th} $ renormalized map $ R^nF_{\mod} $ is also a toy model map, that is,
$$ \pi_{xy} \circ R^nF_{\mod} = R^nF_{2d} $$
for every $ n \in \N $. Moreover, 
$ \eps_{2d, n}(x,y) = \ (b_{1})^{2^n}a(x)\:\! y (1+O(\rho^n)) $ where $ b_{1} $ is the average Jacobian of two dimensional map, $ F_{2d} = \pi_{xy} \circ F_{\mod} $. 
\end{prop}

\ssk
\nin Let $ b_{\mod} $ be the average Jacobian of $ F_{\mod} \in \II(\bar \eps) $. Define another number, $ b_2 $ as the ratio $ b_{\mod}/b_1 $. Then by the above Proposition $ \di_z \de_n \asymp b_2^{2^n} $ for every $ n \in \N $, which is another universal number. Let the following map be a {\em perturbation} of toy model map, $ F_{\mod} (w) = (f(x) - \eps_{2d}(x,y),\;x,\;\de(w)) $ 
\begin{equation}  \label{perturbation of model maps infty}
F(w) = (f(x) - \eps(w),\ x,\ \de(w))
\end{equation}
where $ \eps(w) = \eps_{2d}(x,y) + \widetilde{\eps}(w) $. Thus $ \di_{z} \eps(w) = \di_{z} \widetilde{\eps}(w) $. If $ \| \widetilde{\eps} \| $ is sufficiently small, then $ F $ is called a {\em small perturbation} of $ F_{\mod} $. Let us consider the block matrix form of $ DF $. \msk
\begin{equation} \label{matrix symbolic expression of DF infty}
\begin{aligned}
DF  = 
\renewcommand{\arraystretch}{1.3}
 \left(
\begin{array} {cc|c}
\multicolumn{2}{c|}{\multirow{2}{*}{$D \widetilde {F}_{2d}$}} & \di_{z} \eps  \\ 
 & & { 0} \\    \hline  
\di_x \de & \di_y \de &\di_{z} \de
\end{array}
\right)
= \left( \renewcommand{\arraystretch}{1.3} \begin{array}{c|c}
A & B \\
\hline
C & D
\end{array} \right)
,\quad 
DF_{\mod}  = 
\renewcommand{\arraystretch}{1.3}
 \left(
\begin{array} {cc|c}
\multicolumn{2}{c|}{\multirow{2}{*}{$D {F}_{2d}$}} & 0  \\ 
 & & { 0} \\    \hline  
\di_x \de & \di_y \de &\di_{z} \de
\end{array}
\right)
= \left( \renewcommand{\arraystretch}{1.3} \begin{array}{c|c}
A_1 & {\bf 0} \\
\hline
C & D
\end{array} \right)
\end{aligned} \bsk
\end{equation}
where \ssk $ D\widetilde{F}_{2d} = \begin{pmatrix}
f'(x) - \di_x \eps(w) & -\di_y \eps(w) \\[0.3em]
1  & 0
\end{pmatrix} $\ and\ $ D{F}_{2d} = \begin{pmatrix}
f'(x) - \di_x \eps_{2d}(x,y) & -\di_y \eps_{2d}(xy) \\[0.3em]
1  & 0
\end{pmatrix} $ respectively.
Observe that if $ B \equiv {\bf 0} $, then $ F $ is $ F_{\mod} $. 
Define $ m(A) $ as $ \| A^{-1} \|^{-1} $ and it is called the minimum expansion (or strongest contraction) rate of $ A $.

\ssk
\begin{lem}[Lemma 7.4 in \cite{Nam1}] \label{Invariance of cone filed perturbation infty}
Let $ F $ be a small perturbation of $ F_{\mod} $ defined in \eqref{perturbation of model maps infty}. Let $ A $, $ A_1 $, $ B $, $ C $ and $ D $ be components of the block matrix defined in \eqref{matrix symbolic expression of DF infty}. Suppose that $ \| D \| \leq \frac{\rho_1}{2} \cdot m(A_1) $ for some $ \rho_1 \in (0,1) $. Suppose \ssk also that $ \| B \| \| C \| \leq \rho_0 \cdot m(A) \cdot m(D) $ where $ \rho_0 < \frac{\kappa \gamma}{2} $ for sufficiently small $ \gamma > 0 $. Then there exist the continuous invariant plane field over the given invariant compact set, $ \Gamma $.
\end{lem}

\ssk

\nin The tangent bundle $ T_{\Gamma}B $ has the splitting with subbundles $ E^1 \oplus E^2 $ such that

\ssk
\begin{enumerate}
\item $ T_{\Gamma}B = E^1 \oplus E^2 $. \msk
\item Both $ E^1 $ and $ E^2 $ are invariant under $ DF $. \msk
\item $ \| DF^n |_{E^1(x)} \| \|  DF^{-n} |_{E^2(F^{-n}(x))} \| \leq C \mu^n$ for some $ C>0 $ and $ 0 < \mu < 1 $ and $ n \geq 1 $.
\end{enumerate}
\msk
\nin Then it is called that $ T_{\Gamma}B $ has {\em dominated splitting} over the compact invariant set $ \Gamma $. Moreover, dominated splitting implies that invariant sections 
are continuous by Theorem 1.2 in \cite{New}. Then the maps, $ w \mapsto E^i(w) $ for $ i =1,2 $ are continuous.

\bsk

\section{Single invariant surfaces}

The uniform boundedness of the ratio $ \| D \|  \| A^{-1} \| < \frac{1}{\,2} $ in $ DF $ means that
\begin{equation*}
\sup_{w \in B} \frac{\| D_w \|}{m(A_w)} \leq \frac{1}{2}
\end{equation*}
because the linear operator as the derivative is defined for each point $ w \in B $. It implies the dominated splitting of tangent bundle over a given invariant compact set, $ \Gamma $. If dominated splitting over a given compact set $ \Gamma $ satisfies that
\begin{equation*}
\sup_{w \in B} \frac{\| D_w \|}{m(A_w)^r} \leq \frac{1}{2}
\end{equation*}
for $ r \in \N $, then we say that $ F $ has {\em r-dominated splitting} over $ \Gamma $. Moreover, if $ \| D \| $ for $ DF_{\mod} $ is sufficiently smaller than $ b_1 $ for all $ w \in \Gamma $, then contracting or expanding rates, $ m(A) $ and $ \| D \| $ are separated by a uniform constant over the whole $ \Gamma $. It is called {\em pseudo hyperbolicity}. 

\msk
\subsection{Invariant surfaces and two dimensional ambient space}
Dominated splitting over the given invariant compact set, $ \Gamma $ with smooth cut off function implies the pseudo (un)stable manifolds at each point in $ \Gamma $ tangent to an invariant subbundle. 
However, if the dominated splitting satisfies certain conditions, then the whole compact set is contained in a single invariant submanifold of the ambient space (Theorem \ref{Existence of invariant submanifold} below). 
\msk
\begin{defn}
A $ C^r $ submanifold $ Q $ which contains $ \Gamma $ is {\em locally invariant} under $ f $ if there exists a neighborhood $ U $ of $ \Gamma $ in $ Q $ such that $ f(U) \subset Q $. 
\end{defn}

\nin The necessary and sufficient condition for the existence of these submanifolds, see \cite{CP} or \cite{BC}. 
\begin{thm}[\cite{BC}] \label{Existence of invariant submanifold} Let $ \Gamma $ be an invariant compact set with a dominated splitting $ T_{\Gamma}M = E^1 \oplus E^2 $ such that $ E^1 $ is uniformly contracted. Then $ \Gamma $ is contained in a locally invariant submanifold tangent to $ E^2 $ if and only if the strong stable leaves for the bundle $ E^1 $ intersect the set $ \Gamma $ at only one point.
\end{thm}

\nin Moreover, the existence of invariant submanifold is robust under $ C^1 $ perturbation by \cite{BC}. 
%
Infinitely renormalizable toy model H\'enon-like map with $ b_2 \ll b_1 $ satisfies the sufficient condition for the existence of locally invariant single surfaces by Lemma \ref{transversal intersection at a single point}. 
By $ C^1 $ robustness, the ambient space of toy model maps and its sufficiently small perturbation can be reduced to a single invariant surface.

\begin{rem}
Theorem \ref{Existence of invariant submanifold} is extended to the existence of $ C^r $ invariant submanifold with $ r- $dominated splitting. Moreover, the given invariant compact set can be extended to the maximal one. 
\end{rem}

\begin{lem} \label{invariant surface containing Per and Cantor set}
Let $ F_{\mod} $ be a toy model map in $ \II(\bar \eps) $. Suppose that $ b_2 \ll b_1 $ where $ b_1 $ is the average Jacobian of $ \pi_{xy} \circ F_{\mod} $. Then $ \overline{\Per}_{F_{\mod}} $ has the dominated splitting in Lemma \ref{Invariance of cone filed perturbation infty}. Moreover, there exists a locally invariant $ C^1 $ single surface $ Q $ which contains $ \overline{\Per}_{F_{\mod}} $ and $ Q $ meets transversally and uniquely strong stable manifold, $ W^{ss}(w) $ at each $ w \in \overline{\Per}_{F_{\mod}} $. 
\end{lem}
\begin{proof}
One of the eigenvalues of $ DF_{\mod} $ at each point is asymptotically $ b_2 $ with the eigenvector $ (0\ 0 \ 1 ) $ by straightforward calculation. 
Thus dominated splitting exists with the condition $ b_2 \ll b_1 $ over any invariant compact set, in particular, $ \overline{\Per}_{F_{\mod}} $. Each cone of the vector $ (0\ 0 \ 1 ) $ at all points is disjoint from the invariant plane field, say $ E^{pu} $ - tangent subbundle with pseudo unstable direction. Thus any invariant surface, $ Q $ tangent to $ E^{pu} $ over $ \overline{\Per}_{F_{\mod}} $ meets transversally the strong stable manifold. Let us show the uniqueness of intersection point. Suppose that $ w $ and $ w' $ are intersection points between $ Q $ and $ W^{ss}(w) $. If $ w' \neq w $, then $ w' \notin \overline{\Per}_{\mod} $ by Lemma \ref{transversal intersection at a single point}. Take a small neighborhood $ U $ of $ w' $ in the invariant surface $ Q $. Then $ U $ converges to the neighborhood of $ F^n(w) $ in $ Q $ as $ n \ra \infty $ by Inclination Lemma. Thus $ Q $ cannot be a submanifold of the ambient space because it accumulates itself. It contradicts to Theorem \ref{Existence of invariant submanifold}. Hence, $ w $ is the unique intersection point. 
\end{proof}
\nin Recall that three dimensional H\'enon-like map in $ \II(\bar \eps) $ is sectionally dissipative at each periodic points. Thus the invariant plane field over $ \overline{\Per}_{F_{\mod}} $ contains the unstable direction of each periodic point. Then $ Q $ contains the set
$$ \AAA \equiv \OO \cup \bigcup_{n \geq 1} W^u( \Orb (q_n)) $$
where each $ q_n $ is a periodic point whose period is $ 2^n $ for $ n \in \N $. $ \AAA $ is called the {\em global attracting set}.

\msk
\subsection{Invariant surfaces containing $ \overline{\Per} $ as the graph of $ C^r $ map}
Let $ F_{\mod} $ be the 
H\'enon-like toy model map in $ \II(\bar \eps) $. Let $ b_1 $ be the average Jacobian of $ F_{2d} \equiv \pi_{xy} \circ F_{\mod} $ and assume that $ b_2 \ll b_1 $. 
The set of lines perpendicular to $ xy- $plane
\begin{equation} \label{eq-invariant perpendicular lines}
\begin{aligned}
\bigcup_{(x,\,y) \in\,\pi_{xy}(B)} \{ (x,\, y,\, z ) \,|\; z \in { I}^z \,\}
\end{aligned} 
\end{equation}
is invariant under $ F_{\mod} $. Thus the invariant section, $ w \mapsto E^{ss}(w) $  is constant. The above set, \eqref{eq-invariant perpendicular lines} contains the strong stable manifold over $ \Gamma $. The angle between each tangent spaces $ E^{ss}_w $ and $ E^{pu}_w $ is (uniformly) positive. Thus the maximal angle between $ E^{pu} $ and $ T\R^2 $ is less than $ \frac{\pi}{2} $. 
\begin{rem}
If \,$ T_{\Gamma}B = E^{ss} \oplus E^{pu} $ \,is $ r- $dominated splitting, then $ Q $ which is invariant single surface tangent to $ E^{pu} $ is a $ C^r $ surface. Moreover, since the strong stable manifolds at each point is the set of perpendicular lines to $ xy- $plane, $ Q $ is the graph of $ C^r $ function from a region in $ I^x \times I^y $ to $ I^z $. 
\end{rem}
\nin Let $ F_{\mod} \in \II(\bar \eps) $ with $ b_2 \ll b_1 $. Then by above Lemma \ref{invariant surface containing Per and Cantor set}, we may assume invariant surfaces tangent to the invariant plane field has the neighborhood, say also $ Q $, of the tip, $ \tau_{F_{\mod}} $ in the given invariant single surface which satisfies the following properties. 
\ssk
\begin{enumerate}
\item $ Q $ is contractible. \ssk
\item $ Q $ contains $ \tau_{F_{\mod}} $ in its interior and is locally invariant under $ F^{2^N} $ for big enough $ N \in \N $. \ssk
\item Topological closure of $ Q $ is the graph of $ C^r $ map from a neighborhood of $ \tau \big( \pi_{xy} \circ F_{\mod} \big) $ in $ xy- $plane to $ I^z $. \ssk 
\end{enumerate}
By $ C^1 $ robustness of the existence of single invariant surfaces, let $ F $ be a {\em sufficiently small perturbation} of $ F_{\mod} $ such that there exist invariant surfaces each of which is the graph of $ C^r $ map from a region in the $ xy- $plane to $ I^z $. 
\ssk \\
%
%
%
\begin{prop} \label{invariant surfaces on each deep level}
Let $ F \in \II(\bar \eps) $. Suppose that there exists an invariant surface under $ F $, say $ Q $ which is the graph of $ C^r $ function, $ \xi $ on $ \pi_{xy}(B^n_{\tip}) $ such that $ \| D\xi \| \leq C_0 $ for some $ C_0 >0 $. Then $ Q_n \equiv \big(\Psi^n_{\tip}\big)^{-1}(Q) $ 
is the graph of a $ C^r $ function $ \xi_n $ on $ \pi_{xy}\big(B(R^nF)\big) $ such that 
$$ \xi_n(x,y) = c_0y(1 + O(\si^n)) $$
for some constant $ c_0 $. 
\end{prop}
\begin{proof}
The $ n^{th}$ renormalization of $ F $, $ R^nF $ is $ \big(\Psi^n_{\tip}\big)^{-1} \circ F^{2^n} \circ \Psi^n_{\tip} $. Thus $ Q_n \equiv \big(\Psi^n_{\tip}\big)^{-1}(Q) $ is an invariant surface under $ R^nF $. Let us choose a point 
$ w' = (x',y',z') \in Q \cap B^n_0 $ where $ B_{\tip}^n \equiv \Psi^n_{\tip}(B(R^nF)) $ and $ z' = \xi(x',y') $. Thus 
\begin{align*}
\textrm{graph}(\xi) = (x',\,y',\,\xi(x',y')) =& \ (x',\,y',\,z') .
\end{align*}
Moreover, let $ \big(\Psi^n_{\tip}\big)^{-1}(x',\,y',\,z') = (x,\,y,\,z) \in Q_n $. Thus by the equation \eqref{eq-coordinate change map 3D}, each coordinates of $ \Psi^n_0 \equiv \Psi^n_{\tip}(w-\tau_n) - \tau_F $ as follows
\ssk
\begin{align} 
x' =& \ \alpha_{n,\,0} ( x + S^n_0(w)) +  \si_{n,\,0}\,t_{n,\,0}\cdot y +\si_{n,\,0}\, u_{n,\,0} \,(z + R^n_0(y)) \label{image under the conjugation map xi for x} \\
y' =& \ \si_{n,\,0}\cdot y \label{image under the conjugation map xi for y} \\
z' =& \ \si_{n,\,0}\, d_{n,\,0}\cdot y + \si_{n,\,0}\,(z + R^n_0(y)) \label{image under the conjugation map xi for z} 
\end{align}
where $ w' = (x',\,y',\,z') $. Firstly, let us show that $ Q_n $ is the graph of a well defined function $ \xi_n $ from $ \pi_{xy}(B(R^nF)) $ to $ \pi_z(B(R^nF)) $, that is, $ z = \xi_n(x,\,y) $. 
By the equations \eqref{image under the conjugation map xi for y} and \eqref{image under the conjugation map xi for z}, we see that 
\begin{equation} \label{implicit expression of z}
\begin{aligned} 
\si_{n,\,0}\cdot z = &\ z' - \si_{n,\,0}\,d_{n,\,0}\cdot y - \si_{n,\,0}\, R^n_0(y) \\[0.2em] 
=& \ \xi(x',\,y') - \si_{n,\,0}\,d_{n,\,0}\cdot y - \si_{n,\,0}\, R^n_0(y) \\[0.3em]
=& \ \xi \big(\alpha_{n,\,0} ( x + S^n_0(w)) +  \si_{n,\,0}\,t_{n,\,0}\cdot y + \si_{n,\,0}\,u_{n,\,0}\, (z + R^n_0(y)),\ \si_{n,\,0}\cdot y \big) \\
& \quad - \si_{n,\,0}\,d_{n,\,0}\cdot y - \si_{n,\,0}\, R^n_0(y) .
\end{aligned} \ssk
\end{equation}
Define a function as below \ssk
\begin{equation*}
\begin{aligned}
G_n(x,y,z) =& \ \xi \big(\alpha_{n,\,0} ( x + S^n_0(w)) + \si_{n,\,0} \,t_{n,\,0}\cdot y + \si_{n,\,0}\,u_{n,\,0}\, (z + R^n_0(y)),\ \si_{n,\,0}\cdot y \big) \\
& \quad - \si_{n,\,0}\,d_{n,\,0}\cdot y - \si_{n,\,0}\, R^n_0(y) - \si_{n,\,0}\cdot z .
\end{aligned} \ssk
\end{equation*}
Then the partial derivative of $ G_n $ over $ z $ is as follows \ssk
\begin{equation*}
\begin{aligned}
\di_z G_n(x,y,z) =& \ \di_x \xi \circ \big(\alpha_{n,\,0} ( x + S^n_0(w)) + \si_{n,\,0} \,t_{n,\,0}\cdot y + \si_{n,\,0}\, u_{n,\,0}\,(z + R^n_0(y)),\ \si_{n,\,0}\cdot y \big) \\
& \quad \cdot \big[\, \alpha_{n,\,0} \cdot \di_z S^n_0(w) + \si_{n,\,0}\, u_{n,\,0}\,\big] - \si_{n,\,0}  .
\end{aligned} \msk
\end{equation*}
\nin Recall that $ \alpha_{n,\,0} = \si^{2n}(1 + O(\rho^n)) $, $ \si_{n,\,0} = (-\si)^n (1 + O(\rho^n)) $, $ \| \:\! \di_z S^n_0 \| = O\big(\bar \eps \big) $ and $ | \:\! u_{n,\,0} | = O\big(\bar \eps \big) $. Then
\begin{equation*}
\begin{aligned}
\| \:\! \di_z G_n \| \geq \big[ - \| \:\!\di_x \xi \| \big[\, \si^{2n}C_0\, \bar \eps + \si^n C_1 \;\! \bar \eps \,\big] + \si^n \;\big]\,(1 + O(\rho^n)) 
\end{aligned} \ssk
\end{equation*}
for some positive $ C_0 $ and $ C_1 $. Since $ \| D\xi \| \leq C_0 $ for some $ C_0>0 $, $ \| \:\! \di_z G_n \| $ is away from zero uniformly for small enough $ \bar \eps > 0 $. By implicit function theorem, $ z = \xi_n(x,y) $ is a $ C^r $ function locally on a neighborhood of at every point $ (x,y) \in \pi_{xy}(B(R^nF)) $. Furthermore, since $ Q_n $ is contractible, $ \xi_n(x,y) $ is defined globally by $ C^r $ continuation of the coordinate charts.
\ssk \\
By the equations \eqref{image under the conjugation map xi for y} and \eqref{image under the conjugation map xi for z} with chain rule, we obtain the following equations 
\begin{equation*}
\begin{aligned}
\di_x \xi \cdot \frac{\di x'}{\di x} = & \ \si_{n,\,0} \cdot \di_x \xi_n \\[0.3em]
\di_x \xi \cdot \frac{\di x'}{\di y} + \di_y \xi \cdot \si_{n,\,0} = & \ \si_{n,\,0}\,d_{n,\,0}\:\!  + \si_{n,\,0} \cdot \di_y \xi_n + \si_{n,\,0} \cdot (R^n_0)'(y) .
\end{aligned} 
\end{equation*}
%
\nin Each partial derivatives of $ \xi_n $ as follows by the equation \eqref{image under the conjugation map xi for x}, \ssk
\begin{equation} \label{exponential convergence of the invariant surface xi}
\begin{aligned}
\frac{\di \xi_n}{\di x} = & \ \frac{1}{\si_{n,\,0}} \cdot \di_x \xi \cdot \left[ \alpha_{n,\,0} \big( 1+ \di_{x}S^n_0(w) \big) + \si_{n,\,0}\,u_{n,\,0}\cdot \frac{\di \xi_n}{\di x} \right] \\[0.3em]
\frac{\di \xi_n}{\di y} = & \ \frac{1}{\si_{n,\,0}} \cdot \di_x \xi \cdot \left[ \alpha_{n,\,0}\: \di_{y}S^n_0(w) + \si_{n,\,0}\,t_{n,\,0} + \si_{n,\,0}\,u_{n,\,0} \Big(\: \frac{\di \xi_n}{\di y} + (R^n_0)'( y ) \Big) \right] \\
\qquad & + \di_y \xi - d_{n,\,0} - (R^n_0)'\left( y \right) .
\end{aligned}
\end{equation}
Recall the facts that $ \si_{n,\,0} \asymp (-\si)^n $, $ \alpha_{n,\,0} \asymp \si^{2n} $ for each $ n \in N $
Thus
$$ \Big\| \frac{\di \xi_n}{\di x} \Big\| \leq  \| \di_x \xi \| \, C_0 \:\! \si^n \leq C\:\! \si^n $$
for some $ C>0 $. Recall also that $ \| \di_{y}S^n_0 \| \leq C_3\,\bar \eps $ for some $ C_3 > 0 $ by Lemma \ref{Asymptotic non-linear part}. Each constants $ t_{n,\,0} $, $ u_{n,\,0} $ and $ d_{n,\,0} $ converge to the numbers $ t_{*,0} $, $ u_{*,0} $, and $ d_{*,0} $ respectively super exponentially fast. 
\ssk \\
In the above equation \eqref{exponential convergence of the invariant surface xi}, each partial derivatives $ \di_x \xi $ and $ \di_y \xi $ converges to the origin as $ n \ra \infty $ because all points in the domain of $ \xi $ are in $ B^n_0 \equiv \Psi^n_0(B(R^nF)) $ and $ \diam(B^n_0) \leq C\si^n $. Thus both derivatives $ \di_x \xi(x,y) $ and $ \di_y \xi(x,y) $ converges to $ \di_x \xi(\tau_F) $ and $ \di_y \xi(\tau_F) $ as $ n \ra \infty $ respectively. However, the quadratic or higher order terms of $ \frac{\di \xi_n}{\di y} $ converges to zero exponentially fast by the equation \eqref{eq-C1 convergence of R-n-k}, that is, $ \| R^n_k \|_{C^1} \leq C \si^n $. Hence, we obtain that 
$$ \xi_n(x,y) = c_0y(1 + O(\si^n)) $$  
where $ c_0 = \dfrac{\di_x \xi(\tau_F) \cdot t_{*,0} + \di_y \xi(\tau_F) - d_{*,0}}{1-u_{*,0}} $. 
\end{proof}

\bsk

\section{Universality of conjugated two dimensional H\'enon-like map}
Let $F \in \II(\bar \eps) $ be a sufficiently small perturbation of the given model map $ F_{\mod} \in \II(\bar \eps) $. Let $ Q_n $ and $ Q_k $ be invariant surfaces under $ R^nF $ and $ R^kF $ respectively for $ k < n $. Then by Lemma \ref{invariant surfaces on each deep level}, $ \Psi^n_k $ is the coordinate change map between $ R^kF^{2^{n-k}} $ and $ R^nF $ from level $ n $ to $ k $ such that $ \Psi^n_k(Q_n) \subset Q_k $. Let us define $ C^r $ two dimensional H\'enon-like map $ _{2d}F_{n,\,\xi} $ on level $ n $ \ssk as follows
\begin{align} \label{Cr Henon map with invariant surface}
\,_{2d}F_{n,\, \xi} \equiv \pi_{xy}^{\xi_n} \circ R^nF|_{\,Q_n} \circ (\pi_{xy}^{\xi_n})^{-1} 
\end{align}
\nin where the map $ (\pi_{xy}^{\xi_n})^{-1} : (x,y) \mapsto (x,y, \xi_n(x,y)) $ is a $ C^r $ diffeomorphism on the domain of two dimensional map, $ \pi_{xy}(B) $. In particular, the map $ F_{2d,\,\xi} $ is defined as follows
\begin{equation} \label{Cr Henon map with invariant surface 0}
F_{2d,\:\xi} (x,\,y) = (f(x) - \eps(x,y, \xi),\ x )
\end{equation}
where $ \textrm{graph} (\xi) $ is a $ C^r $ invariant surface under the three dimensional map $ F \colon (x,\, y,\, z) \mapsto (f(x) - \eps(x,y,z),\ x,\ \de(x,y,z) ) $. 

\msk
\subsection{Renormalization of conjugated maps}
Let us assume that $ 2 \leq r < \infty $. By Lemma \ref{invariant surfaces on each deep level}, the invariant surfaces, $ Q_n $ and $ Q_k $ are the graphs of $ C^r $ maps $ \xi_n(x,y) $ and $ \xi_k(x,y) $ respectively. 
The map $ {}_{2d}^{}\Psi^n_{k,\, \xi,\, \tip} $ is defined as the map satisfying the following commutative diagram
\begin{displaymath}
\xymatrix @C=1.8cm @R=1.8cm
 {
(Q_n, \tau_n)  \ar[d]_*+ {\pi_{xy,\, n}^{\xi_n}}  \ar[r]^*+{\Psi^n_{k,{\bf v}, \tip}}    & (Q_k, \tau_k) \ar[d]^*+{\pi_{xy,\, k}^{\xi_k}}   \\
 (_{2d}B_n, \tau_{2d,\,n}) \ar[r]^*+{ {}_{2d}^{}\Psi^n_{k,\, \xi,\, \tip}}      & (_{2d}B_k, \tau_{2d,\,k})
   }
\end{displaymath}
\\ 
where $ Q_n $ and $ Q_k $ are invariant $ C^r $ surfaces with $ 2 \leq r < \infty $ of $R^nF$ and $R^kF$ respectively and $ \pi_{xy,\, n}^{\xi_n} $ and $ \pi_{xy,\, k}^{\xi_k} $ are the inverses of graph maps, $ (x,y) \mapsto (x,y, \xi_n) $ and $ (x,y) \mapsto (x,y, \xi_k) $ respectively. 
\ssk \\
Using translations\, $ T_k : w \mapsto w-\tau_k $\, and\, $ T_n : w \mapsto w-\tau_n $,\, we can let the tip move to the origin as the fixed point of new coordinate change map, $ \Psi^n_k \equiv T_k \circ \Psi^n_{k,\,\tip} \circ T_n^{-1} $. Thus due to the above commutative diagram, corresponding tips in $ _{2d}B_j $ for $ j=k,n $ is changed to the origin. Let $ \pi_{xy} \circ T_j $ be $ T_{2d,\,j} $ for $ j=k,n $. This origin is also the fixed point of the map $ {}_{2d}^{}\Psi^n_{k,\, \xi} := T_{2d,\, k} \circ {}_{2d}^{}\Psi^n_{k,\, \xi,\,\tip} \circ T_{2d,\, n}^{-1} $ where $ T_{2d,\, j} = \pi_{xy,\, j} \circ T_j $ with $ j = k, n $. 
By straightforward calculation, we obtain the expression of $ {}_{2d}^{}\Psi^n_{k,\, \xi} $ as follows
\ssk
\begin{align*}
{}_{2d}^{}\Psi^n_{k,\, \xi} &= \pi_{xy,\, k}^{\xi_k} \circ \Psi^n_k(x,y, \xi_n) \\
& = \pi_{xy,\, k}^{\xi_k} \circ
\left ( \begin{array} {c r r}
\alpha_{n,\,k}   & \si_{n,\,k}\, t_{n,\,k}  &  \si_{n,\,k}\, u_{n,\,k}  \\
                 & \si_{n,\,k}              &                   \\
                 &  \si_{n,\,k}d_{n,\,k}\,  &  \si_{n,\,k}
\end{array} \right )
\left ( \begin{array} {c}
x+ S^n_{k,\, \xi} \\
y \\
\xi_n +R^n_k(y)
\end{array} \right )  \\[0.3em] \numberthis \label{coordinate change map of xi-2d}
& = \left( \alpha_{n,\,k} (x + S^n_{k,\, \xi} ) + \si_{n,\,k} \, t_{n,\,k}\; y +  \si_{n,\,k}\, u_{n,\,k} (\xi_n +R^n_k(y)),\, \si_{n,\,k}\, y \right) \\[-1em]
\end{align*} 
where $ S^n_{k,\, \xi} = S^n_k (x,y, \xi_n(x,y)) $. Then \ssk
\begin{align*} 
\Jac {}_{2d}^{}\Psi^n_{k,\, \xi} = \ & \det \left( \begin{array} {c c}
\alpha_{n,\,k} (1+ \di_x S^n_{k,\, \xi} + \di_z S^n_{k,\, \xi} \cdot \di_x \xi_n ) +  \si_{n,\,k}\, u_{n,\,k}\, \di_x \xi_n & \bullet \\[0.3em]
0 & \si_{n,\,k} 
\end{array} \right) \\[0.6em] \numberthis \label{eq-Jacobian of scope map of xi}
= \ & \si_{n,\,k} \left( \alpha_{n,\,k} (1+ \di_x S^n_{k,\, \xi} + \di_z S^n_{k,\, \xi} \cdot \di_x \xi_n ) +   \si_{n,\,k}\,u_{n,\,k}\, \di_x \xi_n  \right) . \\[-1em]
\end{align*} 
%
If $ F \in \II(\bar \eps) $ has invariant surfaces as the graph of $ C^r $ maps defined on $ I^x \times I^y $ at every level, then $ _{2d}\Psi^{k+1}_{k,\, \xi} $ is the conjugation between 
$ (_{2d}F_{k,\, \xi})^2 $ and $ _{2d}F_{k+1,\, \xi} $ for each $ k \in \N $. Then two dimensional map $ F_{2d, \, \xi} $ is called {\em formally} infinitely {\em renormalizable map} with $ C^r $ conjugation. Moreover, the map defined on the equation \eqref{coordinate change map of xi-2d} with $ n =k+1 $, $ _{2d}\Psi^{k+1}_{k,\, \xi} $ 
is the inverse of the horizontal map 
$$ (x,y) \mapsto (f_k(x) - \eps_k(x,y, \xi_k),\; y) \circ (\si_k x,\ \si_k y)$$
by Proposition \ref{2d scaling map of Cr conjugation} below. 
\msk
\begin{prop} \label{2d scaling map of Cr conjugation}
Let the coordinate change map between $ (_{2d}F_{k,\, \xi})^2 $ and $ _{2d}F_{k+1,\, \xi} $ 
be $ {}_{2d}^{}\Psi^{k+1}_{k,\, \xi} $ which is the conjugation defined on \eqref{coordinate change map of xi-2d}. Then
$$ {}_{2d}^{}\Psi^{k+1}_{k,\, \xi} = H_{k,\, \xi}^{-1} \circ \La_k^{-1} $$

\nin for every $ k \in \N $ where $ H_{k,\; \xi}(x,y) = (f_k(x) - \eps_k(x,y, \xi_{k}),\, y) $ and $ \La_k^{-1}(x,\,y) = (\si_k x,\, \si_k y) $.
\end{prop}
\begin{proof}
Recall the definitions of the horizontal-like diffeomorphism $ H_k $ and its inverse, $ H^{-1}_k $ as follows
\begin{align*}
H_k(w) & = \ (f_k(x) - \eps_k(w),\ y,\ z - \de_k (y, f_k^{-1}(y), 0)) \\
 H^{-1}_k(w) & = \ (\phi_k^{-1}(w),\ y ,\ z + \de_k (y, f_k^{-1}(y), 0)) .
\end{align*}
Observe that $ H_k \circ H^{-1}_k = \id $ and $ f_k \circ \phi_k^{-1}(w) - \eps_k \circ H_k^{-1}(w) = x $ for all points $ w \in \La_k^{-1}(B) $. Then if we choose the set $ \si_k \cdot \text{graph}(\xi_{k+1}) \subset \La_k^{-1}(B) $, then the similar identical equation holds. 
By the definition of $ {}_{2d}^{}\Psi^n_{k,\; \xi} $, the following equation holds
\msk
\begin{align*}
{}_{2d}^{}\Psi^{k+1}_{k,\; \xi}(x,y) &= \ \pi_{xy}^{\xi_k} \circ \Psi^{k+1}_k \circ (\pi_{xy}^{\xi_{k+1}})^{-1}(x,y) \\[0.2em]
&= \ \pi_{xy}^{\xi_k} \circ \Psi^{k+1}_k (x,y, \xi_{k+1}) \\[0.2em]
&= \ \pi_{xy}^{\xi_k} \circ H^{-1}_k \circ \La_k^{-1}(x,y, \xi_{k+1}) \\[0.2em]
&= \ \pi_{xy}^{\xi_k} \circ H^{-1}_k (\si_k x,\,\si_k y,\, \si_k \xi_{k+1}) \\[0.2em]
(*) \quad &= \ \pi_{xy}^{\xi_k}\, \big( \phi_k^{-1}(\si_k x,\,\si_k y,\, \si_k \xi_{k+1}),\, \si_k y,\, \xi_k (\phi_k^{-1} ,\si_k y)\big) \\[0.2em]
&= \ ( \,\phi_k^{-1}(\si_k x,\,\si_k y,\, \si_k \xi_{k+1}),\, \si_k y \, ) . \\[-1.2em]
\end{align*} 
In the above equation, $ (*) $ is involved with the fact that $ H^{-1}_k \circ \La_k^{-1}(\,\text{graph}(\xi_{k+1})) \subset \text{graph}(\xi_{k}) $. 
Let us calculate $ H_{k,\, \xi} \circ {}_{2d}^{}\Psi^{k+1}_{k,\; \xi}(x,y) $. \ssk The second coordinate function of it is just $ \si_k y $. The first coordinate function is as follows
\begin{align*}
& \quad \ \ f_k \circ \phi_k^{-1}(\si_k x,\,\si_k y,\, \si_k \xi_{k+1}) - \eps_k \big( \phi_k^{-1}(\si_k x,\,\si_k y,\, \si_k \xi_{k+1}),\, \si_k y,\, \xi_k (\phi_k^{-1} ,\si_k y)\big) \\
&= \ f_k \circ \phi_k^{-1}(\si_k x,\,\si_k y,\, \si_k \xi_{k+1}) - \eps_k \circ H^{-1}_k (\si_k x,\,\si_k y,\, \si_k \xi_{k+1}) \\
&= \ \si_k x .
\end{align*}
Then $ H_{k,\; \xi}\, \circ \, {}_{2d}^{}\Psi^{k+1}_{k,\; \xi}(x,y) = ( \si_k x,\, \si_k y) $. However, since $  H_{k,\, \xi}\, \circ \,\big( H^{-1}_{k,\, \xi}(x,y) \circ \La_k^{-1}(x,y) \big) = ( \si_k x,\, \si_k y) $, by the uniqueness of inverse map  
$$ {}_{2d}^{}\Psi^{k+1}_{k,\; \xi} = H_{k,\; \xi}^{-1} \circ \La_k^{-1} . $$
\end{proof}
\nin 
Lemma \ref{2d scaling map of Cr conjugation} enable us to define the renormalization of two dimensional $ C^r $ H\'enon-like maps as an extension of the renormalization of analytic two dimensional H\'enon-like maps.

\ssk
\begin{defn} \label{definition of renormalizable Cr Henon map}
 Let $ F : (x,\;y) \mapsto (f(x) -\eps(x,y),\; x) $ be a $ C^r $ H\'enon-like map with $ r \geq 2 $. If $ F $ is renormalizable, then $ RF $, the {\em renormalization} of $ F $ is defined as follows
$$ RF = (\La \circ H) \circ F^2 \circ (H^{-1} \circ \La^{-1}) $$
where $ H(x,y) = (f(x) -\eps(x,y),\; y) $ and the linear scaling map $ \La(x,y) = (sx, sy) $ for the appropriate number $ s < -1 $. 
\end{defn}
\nin If $ F $ is renormalizable $ n $ times, then the above definition can be applied to $ R^kF $ for $ 1 \leq k \leq n $ successively. The two dimensional map $ _{2d}F_{n,\, \xi} $ with the $ C^r $ function $ \xi_n $ is the same as $ R^nF_{2d,\, \xi} $ by Lemma \ref{2d scaling map of Cr conjugation} and the above definition. Thus 
the map $ _{2d}F_{n,\, \xi} $ is realized to be $ R^nF_{2d,\, \xi} $ and called the $ n^{th} $ {\em renormalization} of $ F_{2d,\, \xi} $.
\msk 
\subsection{Universality of conjugated two dimensional maps}
Recall that 
$ \OO_F $ is the same as $ \OO_{F|_Q} $ which is the critical Cantor set restricted to the invariant surface $ Q $. By the $ C^r $ conjugation $ \pi_{xy}^{\xi} $ between $ F|_Q $ and $ F_{2d,\, \xi} $, the ergodic invariant measure on $ \OO_{F_{2d,\, \xi}} $ is defined as the push forward measure $ \mu $ on $ \OO_{F} $ by the map $ \pi_{xy}^{\xi} $, that is, $ (\pi_{xy}^{\xi})_*(\mu) \equiv \mu_{2d,\,\xi}$.  In particular, it is defined as 
\begin{align*}
\mu_{2d,\,\xi} \big(\pi_{xy}^{\xi} (\OO_F \cap B^n_{\bf w}) \big) = \mu_{2d,\,\xi} \big(\pi_{xy}^{\xi}(\OO_F) \cap \pi_{xy}^{\xi}( B^n_{\bf w}) \big) = \frac{1}{\;2^n} .
\end{align*}
Since $ \OO_{F|_Q} $ is independent of any particular surface, so is $ \pi_{xy}^{\xi}(\OO_F) $. Then we express this measure to be $ \mu_{2d} $ because the measure, $ \mu_{2d,\,\xi} $ is also independent of $ \xi $. Let us define the {\em average Jacobian} of $ F_{2d, \, \xi} $
$$
b_{2d} = \exp \int_{\OO_{F_{2d}}} \log \Jac  F_{2d, \, \xi} \; d\mu_{2d} .
$$
This average Jacobian is independent of the surface map $ \xi $ because every invariant surfaces contains the same critical Cantor set, $ \OO_{F_{2d}} $. 

\begin{lem} \label{Universal Jacobian determinant of Cr Henon map}
Let $ F $ be in $ \II(\bar \eps) $ which is a sufficiently small perturbation of toy model map with $ b_1 \gg b_2 $. Suppose that invariant $ C^r $ surfaces $ Q_n $ with $ 2 \leq r < \infty $ under $ R^nF $ contains $ \overline{\Per}_{R^nF} $. Suppose also that $ Q_n =  \text{graph}\,(\xi_n) $ where $ \xi_n $ is $ C^r $ map from $ I^x \times I^y $ to $ I^z $. Let $ R^nF_{2d,\, \xi} $ be $ \pi_{xy}^{\xi_n} \circ F_n|_{\,Q_n} \circ (\pi_{xy}^{\xi_n})^{-1} $ for each $ n\geq 1 $. Then 
$$ \Jac R^nF_{2d,\, \xi} = b_{1,\, 2d}^{2^n} \; a(x) (1+ O(\rho^n))
$$
where $ b_{1,\, 2d} $ is the average Jacobian of \,$F_{2d, \, \xi} $ and $ a(x) $ is the universal function of $ x $ for some positive $ \rho < 1 $.
\end{lem}
\begin{proof}
Lemma \ref{Distortion lemma} could be applied for $ C^r $ H\'enon-like map for $ r \geq 2 $. Thus we obtain 
\begin{align*}
\Jac F^{2^n}_{2d,\, \xi} = b_{1,\, 2d}^{2^n} (1+ O(\rho^n)) .
\end{align*}
Moreover, the chain rule implies that
\begin{align*}
\Jac R^nF_{2d,\, \xi} = b_{1,\, 2d}^{2^n} \; \frac{\Jac {}_{2d} \Psi^n_{0,\,\xi,\,\tip}(x,y)}{\Jac {}_{2d} \Psi^n_{0,\,\xi,\,\tip}(R^nF_{2d,\, \xi}(x,y))} (1+O(\rho^n)) .
\end{align*}
After letting the tip on every level move to the origin by appropriate linear map, the equation \eqref{eq-Jacobian of scope map of xi} implies that 
\begin{align}
\Jac {}_{2d}^{}\Psi^n_{0,\, \xi} = \si_{n,\,0} \Big( \alpha_{n,\,0} \cdot \di_x \big(x + S^n_0(x,y,\,\xi_n)\big) +  \si_{n,\,0} \,u_{n,\,0}\cdot \di_x \xi_n  \Big) .
\end{align}
Then in order to have the universal expression of Jacobian determinant, we need the asymptotic of following maps 
$$ \di_x \big(x + S^n_0(x,y,\,\xi_n)\big) \quad \textrm{and} \quad \dfrac{\si_{n,\,0}}{\alpha_{n,\,0}}\; \di_x \xi_n $$
By Lemma \ref{Asymptotic non-linear part}, 
\begin{align*}
x + S^n_0(x,y,\, \xi_n) = v_*(x) + a_{F,\,1}\: y^2 + a_{F,\,2}\: y \cdot \xi_n + a_{F,\,3}\, (\xi_n)^2 + O(\rho^n) .
\end{align*} 
The above asymptotic has $ C^1 $ convergence with the variable, $ x $. Then
\begin{align*}
\di_x \big(x + S^n_0(x,y,\,\xi_n)\big) = v_*'(x) + a_{F,\,2}\,y\cdot \di_x\xi_n +  2\,a_{F,\,3}\cdot \xi_n\cdot \di_x \xi_n + O(\rho^n) .
\end{align*} \ssk
where $ v_*(x) $ is the universal function for some $ \rho \in (0,1) $. By Proposition \ref{invariant surfaces on each deep level}, we see $ \| \di_x \xi_n \| \leq C \:\! \si^n $.
Then
\begin{align} \label{exponential convergence of Sn with xi}
\di_x \big(x + S^n_0(x,y,\,\xi_n)\big) = v_*'(x) + O(\rho^n) .
\end{align}
By the equation \eqref{exponential convergence of the invariant surface xi} in Proposition \ref{invariant surfaces on each deep level}, 
\msk
\begin{align*}
\frac{\si_{n,\,0}}{\alpha_{n,\,0}} \; \frac{\di \xi_n}{\di x} = & \ \di_x \xi(\bar x, \bar y) \cdot \left[ 1 + \di_x S^n_0(x,y,\,\xi_n) + \frac{\si_{n,\,0}}{\alpha_{n,\,0}} \; u_{n,\,0} \frac{\di \xi_n}{\di x} \right] \\[0.4em]
\textrm{Thus we obtain that}  \qquad \frac{\si_{n,\,0}}{\alpha_{n,\,0}} \; \frac{\di \xi_n}{\di x} = & \ \frac{\di_x \xi(\bar x, \bar y)}{1 - u_{n,\,0}\, \di_x \xi(\bar x, \bar y)} \cdot \big[ 1 + \di_x S^n_0(x,y,\,\xi_n) \big] \phantom{*********} \\[-1em]
\end{align*} 
where $ (\bar x, \bar y) \in \Psi^n_{0, {\bf v}}(B(R^nF_{2d,\,\xi})) $ for all big enough $ n $. Thus $ (\bar x, \bar y) $ converges to the origin as $ n \ra \infty $ exponentially fast by the equation \eqref{eq-diameter of box Bn}. 
\begin{align*} 
\diam(_{2d}\Psi^n_{0,\,\xi}(B)) \leq \diam(\Psi^n_0(B)) \leq C \si^n
\end{align*}
for some $ C>0 $. Recall that the map, $ \di_x \xi(\bar x, \bar y) $ converges to $ \di_x \xi(0,0) $ exponentially fast and $ u_{n,\,0} $ converges to $ u_{*,\,0} $ super exponentially fast. Then
\begin{align}
\frac{\si_{n,\,0}}{\alpha_{n,\,0}} \; \frac{\di \xi_n}{\di x} = \frac{\di_x \xi(0,0)}{1 - u_{*,\,0}\, \di_x \xi(0,0)} \ v_*'(x) + O(\rho^n) .
\end{align}
Let $ (x',y') = R^nF_{2d,\, \xi}(x,y) $. Then
\begin{align} \label{asymptotic of the coordinate change with xi}
\frac{\Jac {}_{2d}^{}\Psi^n_{0,\, \xi}(x,y)}{\Jac {}_{2d}^{}\Psi^n_{0,\, \xi}(x',y')} = 
\frac{1+ \di_x(S^n_{0,\, \xi}(x,y)) + \dfrac{\si_{n,\,0}}{\alpha_{n,\,0}}\;u_{n,\,0}\, \di_x \xi_n(x,y) } {1+ \di_x(S^n_{0,\, \xi}(x',y')) + \dfrac{\si_{n,\,0}}{\alpha_{n,\,0}}\;u_{n,\,0}\, \di_x \xi_n(x',y') }
\end{align}
where $ S^n_0(x,y,\,\xi_n) = S^n_{0,\, \xi}(x,y) $. The translation does not affect Jacobian determinant and each translation from tip to the origin converges to the map $ w \mapsto \tau_{\infty} $ exponentially fast where $ \tau_{\infty} $ is the tip of two dimensional degenerate map, $ F_*(x,y) = (f_*(x),\,x) $. Then by the similar calculation used in Universality Theorem in \cite{CLM}, the equation \eqref{asymptotic of the coordinate change with xi} converges to the following universal function exponentially fast. 
\begin{align*}
\lim_{n \ra \infty} \frac{\Jac {}_{2d}^{}\Psi^n_{0,\, \xi,\,\tip}(x,y)}{\Jac {}_{2d}^{}\Psi^n_{0,\, \xi,\,\tip}(x',y')} &=  \frac{v_*'(x-\pi_x(\tau_{\infty})) + \dfrac{u_{*,\,0}\,\di_x \xi ( \pi_{xy}(\tau_F))}{1- u_{*,\,0}\,\di_x \xi ( \pi_{xy}(\tau_F))}\;v_*'(x-\pi_x(\tau_{\infty})) }{v_*'(f_*(x)- \pi_y(\tau_{\infty})) + \dfrac{u_{*,\,0}\,\di_x \xi ( \pi_{xy}(\tau_F))}{1- u_{*,\,0}\,\di_x \xi ( \pi_{xy}(\tau_F))}\;v_*'(f_*(x)-\pi_y(\tau_{\infty}))} \\[-0.2em]
&= 
\frac{v_*'(x-\pi_x(\tau_{\infty}))}{v_*'(f_*(x)-\pi_y(\tau_{\infty}))} \\[0.4em]
& \equiv a(x) .
\end{align*}
\end{proof}

\begin{thm}[Universality of\; $ C^r $ H\'enon-like maps with $ C^r $ conjugation for $ 2 \leq r < \infty $] \label{Universality of Cr Henon maps}
Let H\'enon-like map $ F_{2d,\, \xi} $ be the $ C^r $ map 
defined on \eqref{Cr Henon map with invariant surface 0} for $ 2 \leq r < \infty $. 
Suppose that $ F_{2d,\, \xi} $ is infinitely renormalizable. Then
\begin{align}
R^nF_{2d,\, \xi}(x,y) = (f_n(x) - (b_{2d})^{2^n}\, a(x)\, y\, (1+ O(\rho^n)),\ x)
\end{align}
where 
$ b_{2d} $ is the average Jacobian of $ F_{2d,\, \xi} $ and $ a(x) $ is the universal function for some $ 0 < \rho < 1 $.
\end{thm}
\begin{proof}
By the smooth conjugation of two dimensional map and $ F_n|_{\,Q_n} $, we see that
\begin{align*}
R^nF_{2d,\, \xi}(x, y) = (f_n(x) - \eps_n(x,y, \xi_n) ,\ x)
\end{align*}
\ssk Denote $ \eps_n(x,y, \xi_n) $ by $ \eps_{n,\,\xi_n}(x,y) $. Then the Jacobian of $ R^nF_{2d,\, \xi} $ is\; $ \di_y \eps_{n,\,\xi_n}(x,y) $. By Lemma \ref{Universal Jacobian determinant of Cr Henon map}, $ \di_y \eps_{n,\,\xi_n}(x,y) = (b_{2d})^{2^n} \, a(x) (1+ O(\rho^n)) $. Then 
$$ \eps_{n,\,\xi_n}(x,y) = (b_{2d})^{2^n}\, a(x)\, y\, (1+ O(\rho^n)) + U_n(x) . $$ 
The map $ U_n(x) $ which depends only on the variable $ x $ can be incorporated to $ f_n(x) $. 
\end{proof}
\nin Recall that the conjugation between $ R^nF_{2d,\, \xi,\,\tip} $ and $ \big( R^kF_{2d,\, \xi} \big)^{2^{n-k}} $ is $ {}_{2d}^{}\Psi^n_{k,\, \xi} $. Recall also that $ \si_{n,\,k} = (-\si)^{n-k} (1 + O(\rho^k)) $ and\, $ \alpha_{n,\,k} = \si^{2(n-k)} (1 + O(\rho^k)) $.

\msk
\begin{thm} \label{Universal estimation of scaling maps}
Let 
$ R^kF \in \II(\bar \eps^{2^k}) $ be the map which has invariant surfaces $ Q_k \equiv \text{graph}(\xi_k) $ tangent to $ E^{pu} $ over the critical Cantor set. Then the coordinate change map, $ {}_{2d}^{}\Psi^n_{k,\, \xi} $ is expressed as follows
\begin{equation} \label{scaling of Cr Henon maps}
\begin{aligned}
{}_{2d}^{}\Psi^n_{k,\, \xi}(x,y) = & \ \big(\,\alpha_{n,\,k}\, (\,x +\, _{2d}S^n_k(x,y)) + \si_{n,\,k}\cdot {}_{2d}t_{n,\,k}\cdot y,\ \si_{n,\,k}\,y\,\big)
\end{aligned} \ssk
\end{equation}
where $ x+{}_{2d}S^n_k(x,y) $ has the asymptotic
$$ x + {}_{2d}S^n_k(x,y) = v_*(x) + a_{F,\:k}\, y^2 + O(\rho^{n-k}) $$  
for $ |\;\!a_{F,\:k}| = O(\eps^{2^k}) $ and $ \rho \in (0,1) $. 
\end{thm}
\begin{proof}
By Lemma \ref{2d scaling map of Cr conjugation}, the coordinate change map, $ {}_{2d}^{}\Psi^n_{k,\, \xi} $ is the composition of the inverse of horizontal diffeomorphisms with linear scaling maps as follows
$$ H_{k,\; \xi}^{-1} \circ \La_k^{-1} \circ H_{k+1,\; \xi}^{-1} \circ \La_{k+1}^{-1} \circ \cdots \circ H_{n,\; \xi}^{-1} \circ \La_n^{-1} . $$ 
Then after reshuffling non-linear and linear parts separately by direct calculations and letting the tip move to the origin by appropriate translations on each levels, the coordinate change map is of the form in \eqref{scaling of Cr Henon maps}. In order to estimate $ _{2d}S^n_k(x,y) $, the recursive formulas of the first and the second partial derivatives of $ _{2d}S^n_k(x,y) $ are required. 
However, the calculation in Section 7.2 in \cite{CLM} can be used because analyticity does not affect any recursive formulas of derivatives and furthermore it just requires $ C^r $ map for $ r\geq 2 $. Hence, recursive formulas with same estimations are applied to $ _{2d}S^n_k(x,y) $. Thus we have the following estimation
$$ x + \,_{2d}S^n_k(x,y) = v_*(x) + a_{F,\;k}\,y^2 + O(\rho^{n-k}) $$  
where $ |\;\! a_{F,\;k}| = O(\eps^{2^k}) $. Alternatively, let us choose the equation \eqref{coordinate change map of xi-2d}
$$ {}_{2d}^{}\Psi^n_{k,\, \xi}(x,y) = \left( \alpha_{n,\,k} (x + S^n_{k,\, \xi}(x,y) ) + \si_{n,\,k} \, t_{n,\,k}\; y +  \si_{n,\,k}\, u_{n,\,k} (\xi_n +R^n_k(y)),\, \si_{n,\,k}\, y \right)  $$
where $ S^n_{k,\, \xi}(x,y) = S^n_k (x,y, \xi_n(x,y)) $. By Proposition \ref{invariant surfaces on each deep level}, the map
$$ \xi_n(x,y) = c_0y + \eta(y) + O(\rho^n) $$ 
where the map $ \eta(y) $ is quadratic or higher order terms with $ \| \eta \|_{C^1} \leq C_0 \si^{n-k} $ for some $ C_0>0 $. By equations \eqref{eq-estimation of t,u,d} and \eqref{eq-C1 convergence of R-n-k}, $ |u_{n,\,k}| \leq C_1 \bar\eps^{2^k} $ and $ \| R^n_k \|_{C^1} \leq C_2 \si^{n-k} $ for some positive $ C_1 $ and $ C_2 $. Recall that the constants, $ \alpha_{n,\,k} = \si^{2(n-k)}(1+O(\rho^n)) $ and $ \si_{n,\,k} = (-\si)^{n-k}(1+O(\rho^n)) $. Hence, we appropriately define each terms of $ {}_{2d}^{}\Psi^n_{k,\, \xi} $ 
\begin{align*}
 {}_{2d}S^n_k(x,y) &= S^n_{k,\, \xi}(x,y) + \frac{\si_{n,\,k}}{\alpha_{n,\,k}}\,u_{n,\,k} [\,\xi_n(x,y) - c_0y + R^n_k(y) ] \\
 {}_{2d}t_{n,\,k} &= t_{n,\,k} + u_{n,\,k} c_0
\end{align*}
which are as desired. 
\ssk \end{proof}

\nin Let $ {}_{2d}t_{k+1,\,k} $ be $ {}_{2d}t_{k} $ for simplicity. Similarly, denote $ \alpha_{k+1,\,k} $ and $ \si_{k+1,\,k} $ to be $ \alpha_k $ and $ \si_k $ respectively. The following corollary and the proof is the same as those of analytic maps in \cite{CLM}. For the sake of completeness, the proof is written below. 
\begin{cor} \label{Estimation of tilt,t-k}
Let $ F_{2d,\,\xi} $ be the infinitely renormalizable $ C^r $ H\'enon-like map with single invariant surfaces tangent to $ E^{pu} $ over the critical Cantor set. Let $ {}_{2d}S^n_k $ be the coordinate change map between $ R^kF_{2d,\,\xi} $ and $ R^nF_{2d,\,\xi} $ defined in Theorem \ref{Universal estimation of scaling maps}. Then 
$$ t_k \asymp - (b_{2d})^{2^k} $$
for every $ k \in \N $. 
\end{cor}
\begin{proof}
Compare the derivative of $ \La_k \circ H_{k,\; \xi} $ at the tip and the derivative of $ \big({}_{2d}^{}\Psi^{k+1}_{k,\, \xi} \big)^{-1} $ at the origin as follows
\begin{align*} \label{eq-derivative of Psi at origin}
&\begin{pmatrix}
1 & - _{2d}t_{k} \\
& 1 
\end{pmatrix} = 
\begin{pmatrix}
\alpha_{k} & \\
 & \si_{k}
\end{pmatrix}
\begin{pmatrix}
\bullet & - s_k \cdot \di_y \eps_{n,\,\xi_n}(\tau_k)  \\
0 & 1
\end{pmatrix}
\end{align*}
Thus $ _{2d}t_{k} = \alpha_{k} \cdot s_k \cdot \di_y \eps_{n,\,\xi_n}(\tau_k) $ where $ s_k \asymp -1 $. Since by Lemma \ref{Universal Jacobian determinant of Cr Henon map}, 
$$ - \di_y \eps_{n,\,\xi_n}(\tau_k) \asymp - \Jac R^nF_{2d,\,\xi} \asymp - (b_{2d})^{2^k} . $$
Then $ \,\!{}_{2d}t_{k} \asymp - (b_{2d})^{2^k} $ for each $ k \in \N $. 
\end{proof}

\msk

\subsection{Non existence of continuous invariant line field on $ Q_n $}

\begin{lem} \label{No continuous invariant line field of Cr Henon map}
Let $ F_{2d,\, \xi} $ be a $ C^r $ infinitely renormalizable two dimensional H\'enon-like map for $ 2 \leq r < \infty $. 
Then $ F_{2d,\, \xi} $ has no continuous invariant line field over the critical Cantor set. Especially, every invariant line fields are discontinuous at the tip.
\end{lem}
\begin{proof}
Universality Theorem \ref{Universality of Cr Henon maps} and the estimation of scaling map, $ \Psi^n_k $ in Theorem \ref{Universal estimation of scaling maps} imply the universal expression of H\'enon-like maps and of horizontal map similar to those of analytic ones. Then the proof discontinuity of invariant line field is essentially the same as the proof of Theorem 9.7 in \cite{CLM}. 
\end{proof}

\begin{thm}
\label{Discontinuity of invariant line field for small perturbation}
Let $ F \in \II(\bar \eps) $ be a sufficiently small perturbation of toy model map with $ b_2 \ll b_1 $. Let $ Q $ be an invariant surface under $ F $ which is tangent to the continuous invariant field, say $ E $, over $ \OO_F $. Then any invariant line field in $ E $ over $ \OO_F $ is discontinuous at the tip.
\end{thm}
\begin{proof}
The proof is the same as that of Theorem 7.8 in \cite{Nam1} with the above Lemma \ref{No continuous invariant line field of Cr Henon map}.
\end{proof}


\nin 
The geometric properties of critical Cantor set --- non existence of continuous invariant line field 
and unbounded geometry of critical Cantor set --- are showed in the invariant surface. These negative results on the invariant surfaces are also valid on three dimensional analytic H\'enon-like maps in no time.


\bsk

\section{Density of conjugated maps in $ C^r $ H\'enon-like maps}
The renormalization for analytic H\'enon-like map is extended to $ C^r $ H\'enon-like maps by invariant $ C^r $ single surfaces of analytic three dimensional map. We would show that the set of $ C^r $ H\'enon-like maps from invariant surfaces is open and dense in $ C^r $ infinitely renormalizable H\'enon-like maps in the parameter space of average Jacobian for any given $ 2 \leq r < \infty $ (Theorem \ref{thm-open dense subset of Cr Henon-like maps}).

\begin{lem} \label{lem-existence of Cr invariant surface, openness}
Let $ F_{\mod} \in \II(\bar \eps) $ be the infinitely renormalizable toy model three dimensional H\'enon-like map. Assume that $ b_2 \ll b_1 $ and there exist invariant $ C^r $ single surfaces which are tangent to $ E^{pu} $ over the critical Cantor set, $ \OO_{F_{\mod}} $ and these surfaces is the graph of $ C^r $ map from $ I^x \times I^y $ to $ I^z $. Let a sufficiently small perturbation of $ F_{\mod} $ with parameter $ t $ as follows
\begin{equation} \label{eq-small perturbation of 3d map with t}
F_t(x,y,z) = (f(x) - \eps(x,y) +tz,\ x,\ \de(x,y,z)) 
\end{equation}
for small enough $ |t| $. Then $ F_t $ has also invariant $ C^r $ single surfaces tangent to $ E^{pu} $ over its critical Cantor set.
\end{lem}
\begin{proof}
The existence of invariant cone fields of $ DF_{\mod} $ and a small perturbation of $ DF_{\mod} $ by Lemma 7.3 and Lemma 7.4  in \cite{Nam1}. Existence of single invariant surfaces for $ F_{\mod} $ is due to Section 3.   
\end{proof}
\nin Denote an invariant single surface of $ F_t $ by $ graph(\xi_t) $ where $ \xi $ is the $ C^r $ map from $ I^x \times I^y $ to $ I^z $. Thus the $ C^r $ H\'enon-like map from invariant surface, $ \pi_{xy} \circ F_t \,|_{graph(\xi_t)} \equiv F_{2d,t} $ is defines as follows
\begin{equation} \label{eq-small perturbation of model map with one parameter}
F_{2d,t} (x,y) = (f(x) - \eps(x,y) + t\:\! \xi_t(x,y),\ x) .
\end{equation}
\nin Let $ F_{2d} $ be a $ C^r $ H\'enon-like map. The unimodal part $ f $ of the following map
$$ F_{2d}(x,y) = (f(x) - \eps^{(r)}(x,y),\ x) $$
can be approximated arbitrary closely by analytic maps in $ C^r $ topology. Then we may assume that $ f $ is analytic and $ \eps^{(r)}(x,y) $ is $ C^r $. Moreover, two variable $ C^r $ map can be also approximated by analytic maps, for instance, multivariate Bernstein polynomials in $ C^{r} $ topology. See \cite{Kin}. Any analytic H\'enon-like maps in $ \II(\bar \eps) $ can be approximated by maps in \eqref{eq-small perturbation of model map with one parameter}. 
%
\begin{lem} \label{lem-a dense subset of Cr Henon-like maps}
The set of two dimensional $ C^r $ H\'enon-like map in \eqref{eq-small perturbation of model map with one parameter} is a dense subset of two dimensional $ C^r $ H\'enon-like map in $ \II(\bar \eps) $ for $ 2 \leq r < \infty $. 
\end{lem}
\begin{lem} \label{lem-Continuous moving of Cantor set with t}
The critical Cantor set, $ \OO_{F_{2d}} $ of two dimensional $ C^r $ H\'enon-like map moves continuously as $ F_{2d} $ in infinitely H\'enon-like maps. 
\end{lem}
\begin{proof}
By construction of the critical Cantor set, for a given word $ {\bf w}_n \in W^n $, the unique periodic point $ w_n $ with period $ 2^n $ of the region $ B^n_{{\bf w}_n} $ is $ C^r $ by Implicit Function Theorem. Each point $ w \in \OO_F $ is the limit of $ w_n $ as $ n \ra \infty $ for the given word $ {\bf w} \in W^{\infty} $ which contains $ {\bf w}_n $ as a finite subaddress of $ {\bf w} $ for every $ n \in \N $. Since two dimensional box, $ B^n_{{\bf w}_n}(F_{2d}) $ is $ \pi_{xy}\big(B^n_{{\bf w}_n}(F) \big) $ of three dimensional map $ F $, the uniform convergence of three dimensional boxes as $ n \ra \infty $ implies that of two dimensional ones. Then the critical Cantor set moves continuously  as $ F_{2d} $.  
\end{proof}
\msk
\nin Recall the maps in \eqref{eq-small perturbation of 3d map with t} and \eqref{eq-small perturbation of model map with one parameter} for $ |t| < r $ where $ r $ is sufficiently small such that 
\begin{enumerate}
\item For every $ |t| < r $, there exist single invariant surfaces tangent to $ E^{pu} $ over the critical Cantor set as the graph from $ I^x \times I^y $ to $ I^z $. \ssk
\item $ \Jac F_{2d,\,t} $ is positive on $ (-r, r) \times B $. 
\end{enumerate}
\begin{cor} \label{cor-continuous moving of average Jacobian}
The average Jacobian $ b_{2d,t} \equiv b(F_{2d,t}) $ for $ |t| < r $ moves continuously on $ t $ for sufficiently small $ r>0 $. 
\end{cor}
\begin{proof}
The average Jacobian of $ F_{2d,t} $ is defined explicitly as follows
\begin{equation*}
b_{2d,t} = \exp \int_{\OO_t} \log (\Jac F_{2d,t})\,d\mu_t = \exp \int_{\OO_t} \log \left( \frac{\di \eps}{\di y} + t \frac{\di \xi_t}{\di y} \right)\,d\mu_t
\end{equation*}
where $ \mu_t $ is the unique $ F_{2d,t} $-invariant probability measure on each critical Cantor set $ \OO_t \equiv \OO_{F_{2d,t}} $. By Lemma \ref{lem-Continuous moving of Cantor set with t}, $ \OO_t $ moves continuously. Then the integral is also continuous with $ t $. 
\end{proof}
\begin{rem}
If the H\'enon-like map $ F_t $ in $ \II(\bar \eps) $ is analytic and it is extendible holomorphically, then the critical Cantor set moves holomorphically with $ t $ by Lemma 5.6 in \cite{CLM}. 
\end{rem}
\nin Define that a $ C^r $ H\'enon-like map, $ F_{2d} $ is {\em embedded} in analytic {\em three dimensional H\'enon-like map} in $ \II(\bar \eps) $ only if $ F_{2d} $ is conjugated by a $ C^r $ map to $ F|_Q $ where $ Q $ is a $ C^r $ invariant surface tangent to $ E^{pu} $ over the critical Cantor set. 
\begin{thm} \label{thm-open dense subset of Cr Henon-like maps}
Let $ F_{2d,b} $ be an element of parametrized $ C^r $ H\'enon-like maps for $ b \in [0,1) $ in $ \II(\bar \eps) $ where $ b $ is the average Jacobian of $ F_{2d,b} $ for $ 2 \leq r < \infty $. Then for some $ \bar b > 0 $, the set of parameter values, an interval $ [0,\bar b] $ on which the map $ F_{2d,b} $ is embedded in three dimensional analytic H\'enon-like maps in $ \II(\bar \eps) $ contains a dense open subset. 
\end{thm}
\begin{proof}
The density of the set of conjugated map from invariant surfaces is due to Lemma \ref{lem-a dense subset of Cr Henon-like maps}. The openness is involves with Lemma \ref{lem-existence of Cr invariant surface, openness} and Corollary \ref{cor-continuous moving of average Jacobian}. 
\end{proof}
\msk

\subsection*{Notes}
The definition of renormalizability of $ C^r $ H\'enon-like map is just extension of that of analytic H\'enon-like maps. However, hyperbolicity of renormalization operator for $ C^r $ H\'enon-like maps at the fixed point is not proved yet. In previous sections, using single invariant surfaces in three dimensional analytic H\'enon-like maps, we construct $ C^r $ conjugation between maps in single invariant surfaces and two dimensional maps. It defines infinite renormalization of $ C^r $ H\'enon-like maps in this class.  
Moreover, direct calculations of asymptotics in \cite{CLM} to this article, the smoothness of invariant surfaces seems to be sufficient for $ r =2 $. However, 
the hyperbolicity of period doubling operator of one dimensional maps requires $ C^{2 + \epsilon} $ maps with arbitrary small but positive number $ \epsilon $ in \cite{Dav} and moreover, H\'enon renormalization contains that of one dimensional maps as degenerate maps. On the other hand, since invariant surfaces are constructed by invariant cone fields, these surfaces cannot be $ C^{\infty} $ or analytic. Existence of any single invariant $ C^{\infty} $ or non-flat analytic surfaces tangent to $ E^{pu} $ over the critical Cantor set is not known yet.


\bsk

\section{Unbounded geometry on the Cantor set}
Let the subset of critical Cantor set on each pieces be $ \OO_{\bf w} \equiv B^n_{\bf w} \cap \OO $ where $ {\bf w} \in W^n = \{ v,\,c\}^n $ is the word of length $ n $. We may assume that every box region is (path) connected and simply connected. Suppose that each topological region, $ B^n_{\bf w} $ compactly contains $ \OO_{\bf w} $ and moreover $ \overline{B^n_{\bf w}} $ is disjoint from $ \overline{\OO \setminus \OO_{\bf w}} $ for every word $ \bf w $. Assume also that every $ B^n_{\bf w} $ is forward invariant under $ F^{2^n} $ for all word $ \bf w $ and every $ n \in \N $. {\em Bounded geometry} is defined for given box regions which satisfy the following
\begin{align*}
 \dist_{\min}(B^{n+1}_{{\bf w}v}, B^{n+1}_{{\bf w}c}) &\asymp \diam(B^{n+1}_{{\bf w}\nu}) \quad \text{for} \ \nu \in \{v, c\} \\[0.2em]
 \diam(B^n_{\bf w}) &\asymp \diam(B^{n+1}_{{\bf w}\nu}) \quad \text{for} \ \nu \in \{v, c\}
\end{align*}
for all $ {\bf w} \in W^n $ and for all $ n \geq 0 $. 
%
\nin The proof of unbounded geometry of critical Cantor set requires to compare the diameter of boxes and the minimal distance of two adjacent boxes. In order to compare these quantities, we would use the maps, $ \Psi^n_k $, $ R^kF $ and $ \Psi^k_0 $ with the two points $ w_1 = (x_1,\, y_1,\, z_1) $ and $ w_2 = (x_2,\, y_2,\, z_2) $ in $ \OO_{R^nF} $. Let us each successive image of $ w_j $ under $ \Psi^n_k $, $ R^kF $ and $ \Psi^k_0 $ be $ \dot w_j $, $ \ddot w_j $ and $ \dddot w_j $ for $ j=1,2 $.
\begin{displaymath}
\xymatrix  
{   { w_j}  \ar@{|->}[r]^ {\Psi^n_k} & \dot w_j  \ar@{|->}[r]^{R^kF}  &\ddot w_j  \ar@{|->}[r]^{ \Psi^k_0}     & \dddot w_j 
   }
\end{displaymath}
Let the coordinates of the point, $ \dot w_j $ be $ (\dot x_j ,\, \dot y_j ,\, \dot z_j ) $. The points $ \ddot w_j  $ and $ \dddot w_j $ also have the similar coordinate expressions. Let $ S_1 $ and $ S_2 $ be the (path) connected set on $ \R^3 $. If $ \pi_x(\overline{S}_1) \cap \pi_x(\overline{S}_2) $ contains at least two points, then this intersection is called the {\em $ x- $axis overlap} or {\em horizontal overlap} of $ S_1 $ and $ S_2 $. Moreover, we say $ S_1 $ {\em overlaps} $ S_2 $ on the $ x- $axis or horizontally.
\ssk \\
\nin Let $ F_{2d} $ be an infinitely renormalizable two dimensional H\'enon-like map and $ b_1 $ be the average Jacobian of $ F_{2d} $. Then unbounded geometry of the critical Cantor set depends on Universality theorem and the asymptotic of the tilt, $ - t_k \asymp b_1^{2^k} $ but it does not depend on the analyticity of the map. The infinitely renormalizable $ C^r $ H\'enon-like maps defined by invariant surfaces has Universality by Theorem \ref{Universality of Cr Henon maps} and the asymptotic of the tilt $ - {}_{2d}t_k \asymp b_1^{2^k} $ by Corollary \ref{Estimation of tilt,t-k}. Then unbounded geometry of the critical Cantor set in \cite{CLM} and \cite{HLM} is applicable to $ C^r $ H\'enon-like map defined by invariant surfaces.
\ssk \\
Observe that $ \dist_{\min}(S_1, S_2) \leq \dist(w_1, w_2) $ for all $ w_1 \in S_1 $ and $ w_2 \in S_2 $ and $ \diam(S) \geq \dist(w, w') $ for all $ w, w' \in S $.
\begin{lem}  \label{bounds of min distance and diameter}
Let $ F_{2d} $ be an infinitely renormalizable $ C^r $ H\'enon-like maps defined by invariant surfaces which is tangent to $ E^{pu} $ over $ \OO_F $. 
Suppose that two dimensional box $ _{2d}B^{n-k}_{{\bf v}v}(R^kF_{2d}) $ overlaps\, $ _{2d}B^{n-k}_{{\bf v}c}(R^kF_{2d}) $ on the $ x- $axis where $ {\bf v} = v^{n-k-1} $. Then for all sufficiently large $ k$ and $ n $ with $ k < n $, 
we have the following estimate
\begin{align*}
\dist_{\min}(_{2d}B^{n}_{{\bf w}v},\, _{2d}B^{n}_{{\bf w}c}) &\leq C_0 \, b_1^{2^k} \si^{2k} \si^{n-k} \\
\diam (_{2d}B^{n}_{{\bf w}v}) &\geq C_1 \si^{2(n-k)} \si^{k}
\end{align*}
where $ {\bf w} = v^kcv^{n-k-1} \in W^n $ for some positive constants $ C_0 $ and $ C_1 $.
\end{lem}
\begin{proof}
The proof is the same as the analytic case because unbounded geometry depends only on the universality theorem and asymptotic of the tilt $ - {}_{2d}t_k \asymp b_1^{2^k} $. Then we can adapt the proof for analytic maps in \cite{HLM}. For the sake of completeness, we describe the proof below.  
Choose two points $ w_1 = (x_1,\,y_1) $ and $ w_2 = (x_2,\,y_2) $ in $ _{2d}B^1_v(R^n\!F_{2d}) \cap \OO_{R^n\!F_{2d}} $ and $ _{2d}B^1_c(R^n F_{2d}) \cap \OO_{R^n\! F_{2d}} $ respectively in order to estimate the minimal distance between two boxes.  \ssk \\
The expression of $ _{2d}\Psi^n_{k,\,\xi} $ in Theorem \ref{Universal estimation of scaling maps} and overlapping assumption implies the coordinates of the points, $ (\dot x_j,\,\dot y_j) $, $ (\ddot x_j,\,\ddot y_j) $ and $ (\dddot x_j,\,\dddot y_j) $ for $ j=1,2 $ as follows
\begin{equation*}
\dot x_1 - \dot x_2 = 0 \ \ \textrm{and} \ \ \dot y_1 - \dot y_2 = \si_{n,\,k}(y_1 - y_2)
\end{equation*}
The special form of H\'enon-like map, $ R^kF_{2d} $ and coordinate change map, $ _{2d}\Psi^n_{k,\,\xi} $ imply that 
\begin{equation} \label{eq-overlapping on x-axis}
\dddot y_1 - \dddot y_2 = \si_{k,\,0} (\ddot y_1 - \ddot y_2) = \si_{k,\,0} (\dot x_1 - \dot x_2) = 0
\end{equation}
By mean value theorem and the fact that $ (\ddot x_j,\,\ddot y_j) = R^kF_{2d}(\dot x_j,\,\dot y_j) $ for $ j=1,2 $ implies that 
\begin{align*}
\ddot x_1 - \ddot x_2 &= f_k(\dot x_1) - \eps_k(\dot x_1,\,\dot y_1) - \big[ f_k(\dot x_2) - \eps_k(\dot x_2,\,\dot y_2) \big] \\[0.2em]
&= - \eps_k(\dot x_1,\,\dot y_1) + \eps_k(\dot x_2,\,\dot y_2) \\[0.2em]
&= - \di_y \eps_k(\eta) \cdot (\dot y_1 - \dot y_2) \\[0.2em]
&= - \di_y \eps_k(\eta) \cdot \si_{n,\,k}(y_1 - y_2)
\end{align*}
where $ \eta $ is some point in the line segment between $ (\dot x_1,\,\dot y_1) $ and $ (\dot x_2,\,\dot y_2) $. Thus by Theorem \ref{Universal estimation of scaling maps} and the equation \eqref{eq-overlapping on x-axis}, we obtain that
\begin{align*}
\dddot x_1 - \dddot x_2 &= \pi_x \circ {}_{2d}\Psi^k_{0,\,\xi}(\ddot x_1,\,\ddot y_1) - \pi_x \circ {}_{2d}\Psi^k_{0,\,\xi}(\ddot x_2,\,\ddot y_2) \\[0.2em]
&= \alpha_{k,\,0}\big[ (\ddot x_1 + {}_{2d}S^k_0(\ddot x_1,\,\ddot y_1)) - (\ddot x_2 + {}_{2d}S^k_0(\ddot x_2,\,\ddot y_2)) \big] + \si_{k,\,0} \big[ {}_{2d}t_{k,\,0} \cdot (\ddot y_1 - \ddot y_2) \big] \\[0.2em]    \numberthis \label{eq-horizontal distance of adjacent boxes}
&= \alpha_{k,\,0}\big[ v'_*(\bar x) + O(\bar \eps + \rho^k)\big] (\ddot x_1 - \ddot x_2) .
\end{align*}
Then by the fact that $ \di_y \eps_k \asymp b_1^{2^k} $ where $ b_{1} $ is the average Jacobian of $ F_{2d} $, we can estimate the minimal distance
\begin{align*}
\dist_{\min}(_{2d}B^{n}_{{\bf w}v},\, _{2d}B^{n}_{{\bf w}c}) &\leq \big| \dddot x_1 - \dddot x_2 \big| + \big|\dddot y_1 - \dddot y_2 \big| \\[0.2em]
 &\leq \si^{2k} \big| \ddot x_1 - \ddot x_2 \big| \cdot  v'_*(\bar x) (1 + O(\rho^k)) \\[0.2em]
 &\leq C_0 \, b_1^{2^k} \si^{2k} \si^{n-k}
\end{align*}
where $ v_*(x) $ is the positive universal function for some $ C_0 > 0 $. Take any two different points, $ (x_1,\,y_1) $ and $ (x_2,\,y_2) $ in the box $ _{2d}B^1_v(R^nF_{2d}) \cap \OO_{R^n\!F_{2d}} $ to estimate the diameter of $ _{2d}B^{n}_{{\bf w}v} $. Thus the special forms of $ R^kF_{2d} $, $ _{2d}\Psi^n_{k,\,\xi} $ and the equation \eqref{eq-horizontal distance of adjacent boxes} implies that
\begin{align*}
\diam(_{2d}B^{n}_{{\bf w}v}) &\geq \big|\dddot y_1 - \dddot y_2 \big| = \si_{k,\,0} \cdot (\ddot y_1 - \ddot y_2) \\[0.2em]
&= \big| \si_{k,\,0} \cdot (\dot x_1 - \dot x_2)\big| \\[0.2em]
&= \big| \si_{k,\,0} \big[ \pi_x \circ {}_{2d}\Psi^n_{k,\,\xi} (\dot x_1,\,\dot y_1) -  \pi_x \circ {}_{2d}\Psi^n_{k,\,\xi} (\dot x_2,\,\dot y_2) \big]\big| \\[0.2em]
&= \big| \si_{k,\,0}\, \alpha_{n,\,k} \big[ v'_*(\widetilde x)(x_1 - x_2) +O(\bar \eps^{2^k} + \rho^{n-k}) \big] \big| \\[0.2em]
&\geq C_1\,\si^{2(n-k)}\si^k
\end{align*}
where $ v_*(x) $ is the positive universal function for some $ C_1>0 $. 
\end{proof}
\msk
\nin Unbounded geometry on the critical Cantor set holds if we choose $ n >k $ such that $ b_1^{2^k} \asymp \si^{n-k} $ for every sufficiently large $ k \in \N $. This is true on the parameter space of average Jacobian, $ b_1 $ almost everywhere with respect to Lebesgue measure. 
\begin{thm}[\cite{HLM}] \label{G-delta subset with full Lebesgue measure on parameter space}
The given any $ 0<A_0<A_1 $, $ 0< \si < 1 $ and any $ p \geq 2 $, the set of parameters $ b \in [0,1] $ for which there are infinitely many $ 0<k<n $ satisfying 
$$ A_0 < \frac{b^{p^k}}{\si^{n-k}} < A_1 $$
is a dense $ G_{\delta} $ set with full Lebesgue measure. 
\end{thm}

\ssk  

\nin Recall that toy model map has universal numbers --- the average Jacobian, $ b_{\mod} $, the average Jacobian of two dimensional map, $ \pi_{xy} \circ F_{\mod} $, $ b_{1,\mod} $ and the ratio of these two numbers, $ b_{2,\mod} \equiv b_{\mod} / b_{1,\mod} $. If $ b_{2,\mod} \ll b_{1,\mod} $, then each of these numbers can be generalized to a sufficiently small perturbation of toy model map. In particular, the number $ b_{1,\mod} $ is generalized to the average Jacobian of $ F_{2d,\,\xi} $, say $ b_1 $, by Theorem \ref{Universality of Cr Henon maps}. Another number $ b_2 $ is just defined as the ratio, $ b_F / b_1 $. 
%
%
%
Then unbounded geometry of Cantor attractor of $ F|_{\,Q} $ on invariant surface is extended to those of same Cantor set for three dimensional map, $ F $.


\begin{thm} \label{Unbounded geometry for model maps}
Let $ F $ be three dimensional H\'enon-like map in $ \II(\bar \eps) $ which is a small perturbation of toy model map with $ b_2 \ll b_1 $. 
Then for each sufficiently small fixed positive number $ b_2 $, the parametrized H\'enon-like map $ F_{b_1} $ for $ b_1 \in [\,b_{\circ}, b_{\bullet}] $ where $ b_1 $ is the average Jacobian of two dimensional $ C^r $ H\'enon-like map, $ F_{2d,\,\xi} $ for $ b_2 \ll b_{\circ} < b_{\bullet} $. Then there exists $ G_{\delta} $ subset $ S $ with full Lebesgue measure of\, $ [\,b_{\circ}, b_{\bullet}] $ such that the critical Cantor set, $ \OO_{F_{b_1}} $ has unbounded geometry.
\end{thm}
\begin{proof}
The box on the invariant surface $ Q $, $ _{Q}B^{n}_{{\bf w}} $ is defined as the image of the box, $ _{2d}B^{n}_{{\bf w}} $ of two dimensional H\'enon-like map under the graph map $ (x,y) \mapsto (x,y, \xi) $ for every $ n \in \N $ and every word $ {\bf w} \in W^n $. By Proposition \ref{invariant surfaces on each deep level}, the minimal distance between two boxes on the surface and that between two boxes on $ xy- $plane with the same word are comparable with each other for all words. Furthermore, there exist three dimensional boxes, $ B^{n}_{{\bf w}} $ such that $ Q \cap B^{n}_{{\bf w}} \supset {}_{Q}B^{n}_{{\bf w}} $ for every word $ {\bf w} $ because $ Q $ is an invariant surface which compactly contains the critical Cantor set. Then by Lemma \ref{bounds of min distance and diameter}, we have
\begin{equation*}
\begin{aligned}
\dist_{\min}(_{2d}B^{n}_{{\bf w}v},\; _{2d}B^{n}_{{\bf w}c}) &\asymp \dist_{\min}(_{Q}B^{n}_{{\bf w}v},\; _{Q}B^{n}_{{\bf w}c}) \\
\dist_{\min}(B^{n}_{{\bf w}v},\; B^{n}_{{\bf w}c}) &\leq \dist_{\min}(_{Q}B^{n}_{{\bf w}v},\; _{Q}B^{n}_{{\bf w}c}) \leq C_0\, b_1^{2^k} \si^{2k} \si^{n-k}
\end{aligned} \msk
\end{equation*}
for the word $ {\bf w} = v^{n-k-1}cv^k $ and moreover, \ssk
\begin{equation*}
\begin{aligned}
\diam (_{2d}B^{n}_{{\bf w}v}) &\asymp \diam (_{Q}B^{n}_{{\bf w}v}) \\
\diam (B^{n}_{{\bf w}v}) &\geq \diam (_{Q}B^{n}_{{\bf w}v}) \geq C_1 \si^{2(n-k)} \si^{k}
\end{aligned} \ssk
\end{equation*}
\nin for the word $ {\bf w} = v^{n-k-1}cv^k $ and for positive constants $ C_0 $ and $ C_1 $ independent of $ \bf w $ and $ n $. 
One box overlaps its adjacent box on the $ x- $axis in three dimension if and only if so does in two dimension because there exists an invariant surface as the graph from the plane to $ z- $axis. Then 
\begin{equation*}
 b_1^{2^k} \asymp \si^{n-k}
\end{equation*}
for all sufficiently large $ k $ in the $ G_{\delta} $ subset which has full measure in the parameter space\, $ [\,b_{\circ}, b_{\bullet}] $ by Theorem \ref{G-delta subset with full Lebesgue measure on parameter space}. Hence, $ \dist_{\min}(B^{n}_{{\bf w}v},\, B^{n}_{{\bf w}c}) \leq  C \si^{k}\diam (B^{n}_{{\bf w}v}) $ for some $ C>0 $. Therefore, the critical Cantor set has unbounded geometry.
\end{proof}

\bsk

\titleformat{\section}[display]{\normalfont\Large\bfseries}{Appendix~\Alph{section}}{12pt}{\Large}

\begin{appendices}

\section{Periodic points and critical Cantor set}
Let us take a word, $ {\bf w} = (w_1\,w_2\,w_3 \ldots w_n \ldots )$ as an address. The word of the first $ n $ concatenations, $ {\bf w}_n = (w_1\,w_2\,w_3 \ldots w_n ) $ is defined as the {\em subaddress} of the word $ {\bf w} $.
\ssk

\begin{lem} \label{limit points of periodic points}
Let $ F $ be the H\'enon-like map in $ \II(\bar \eps) $ with sufficiently small positive $ \bar \eps $. Then the set of accumulation points of $ \Per_F $ is the critical Cantor set $ \OO_F $. 
\end{lem}
\begin{proof}
The region $ B^n_{{\bf w}_n} \equiv \Psi^n_{0, {\bf w}_n}(B(R^nF)) $ contains the periodic point, $ \Psi^n_{0, {\bf w}_n}(\beta_1(R^nF)) $ with period $ 2^n $. By construction of the critical Cantor set, every point $ \OO_F $, say $ w $ is as follows
$$ \{w\} = \bigcap_{n \geq 0} B^n_{\bf w_n} $$
for the corresponding words, $ {\bf w}_n $ are the subaddresses of $ {\bf w} \in W^{\infty} \equiv \{v, c\}^{\infty} $ for all $ n \in \N $. Since $ \diam (B^n_{{\bf w}_n}) \leq C\si^n $ for all word $ {\bf w}_n $ and for all $ n \in \N $, every points in $ \OO_F $ is contained in the set of accumulation points of $ \Per_F $. 
For the reverse inclusion, recall the following facts 
\begin{itemize}
\item[---] For any H\'enon-like map $ F \in \II(\bar \eps) $, the region $ B^1_v \cup B^1_c $ contains all periodic points of $ F $. \ssk
\item[---] The number of periodic points with any given single period, $ 2^n $ is always finite. \ssk
\item[---] The region $ B^N_{{\bf w}_N} $ compactly contains $ B^n_{{\bf w}_n} $ where $ n>N $ and the word $ {\bf w}_N $ is a subaddress of the word $ {\bf w}_n $. 
\end{itemize}
Take any point, say $ w $, in the set of accumulation point of $ \Per_F $. We may assume that there exists a sequence of periodic points, $ \{ q_{n_k} \} $ which converge to $ w $ as $ k \ra \infty $ where the period of each $ q_{n_k} $ is $ 2^{n_k} $ and $ {n_k} $ is increasing and $ n_k \ra \infty $ as $ k \ra \infty $. Observe that the periodic point $ q_{n_k} $ is $ \Psi^{n_k}_{0, {\bf w}_{n_k}}(\beta_1(R^{n_k}F)) $ for some address $ {\bf w}_{n_k} $. We claim that there exists a periodic point, $ q_{n_k} $ of which region $ B^{n_k}_{{\bf w}_{n_k}} $ contains $ w $. If not, then $ \Orb_F \big(\overline{ B^{n_k}_{{\bf w}_{n_k}}} \big) $ is disjoint from $ w $. However, every periodic points of which period is greater than $ q_{n_k} $ are in $ \Orb_F \big(\overline{ B^{n_k}_{{\bf w}_{n_k}}} \big) $. It contradicts the convergence of periodic points to $ w $.  Then we may assume that the region $ B^{n_k}_{{\bf w}_{n_k}} $ contains $ w $ and the sequence $ Q \equiv \{q_{n_m} \;|\; m >k \} $. Denote the region $ B^{n_k}_{{\bf w}_{n_k}} $ by $ B_k $ for each $ k $. Since every points $ q_{n_m} \in Q $ are a periodic points under $ R^{n_k}F $ in $ B(R^{n_k}) $, each region, $ B_m $ for $ m>k $ is compactly contained in $ B_k $ and moreover, $ B_m $ converges to $ w $ as $ m \ra \infty $. Each region $ B_m $ has its own address and the address converges to a word $ {\bf w} \in W^{\infty} $ as $ m \ra \infty $. This construction implies that the sequence of $ B_m $ converges to a point with the address $ \bf w $ in the critical Cantor set. Hence, the accumulation point, $ w $ is contained in $ \OO_F $. 
\end{proof}
\ssk

\begin{lem} \label{transversal intersection at a single point}
Let $ F $ be the three dimensional H\'enon-like map in $ \II(\bar \eps) $ for small enough $ \bar \eps > 0 $. Then $ W^{s}(w) \cap \overline\Per_{F} = \{w\} $ for each $ w \in \overline\Per_{F} $.
\end{lem}
\begin{proof}
The fact that $ F \in \II(\bar \eps) $ implies the existence of the critical Cantor set. Note that any given periodic points of $ F $ has period, $ 2^k $ for some $ k \in \N $. For any two periodic points, $ p $ and $ q $, we may assume that these points are fixed points under $ F^{2^k} $ for large enough $ k \in \N $. If both $ p $ and $ q $ are in any same stable manifold, then $ \dist (F^n(p), F^n(q)) \ra 0 $ as $ n \ra \infty $. However, $ \dist (F^{2^{km}}(p), F^{2^{km}}(q)) $ is fixed for every $ m \in \N $. Thus $ p $ is the same as $ q $. 
\ssk \\
Any point $ w $ in the critical Cantor set has its address of which length is infinity and the sequence of boxes containing $ w $ with the address which is the first finite concatenations of the address of $ w $. Thus 
each point in the critical Cantor set is the limit of box domain, that is, $ \displaystyle{ \{w \} = \bigcap_{N \geq 0}B^N_{ {\bf w}_N} } $ where $ {\bf w}_N $ is the subaddress of $ w $ for all $ N \in \N $. Since $ B^N_{{\bf w}_N} $ are forward invariant under $ F^{2^{N+1}} $, for any given periodic point, say $ q $ both the box domain $ B^N_{ {\bf w}_N} $ and $ q $ are invariant under $ F^{2^{N+1}} $ for all big enough $ N $. Moreover, due to the fact that $ \diam (B^N_{{\bf w}_N}) = C\si^n $ for some $ C>0 $, we may assume that $ B^N_{ {\bf w}_N} $ is disjoint from $ \{ q\} $. Then $ \dist (F^{2^{N}m}(q), F^{2^{N}m}(w)) \geq c_0 $ for all $ m \geq 2 $ and for some $ c_0 > 0 $. 
Then $ W^{s}(w) $ for each $ w \in \OO_F $ does not contain any other point in $ \Per_F $. Similarly, $ W^{s}(\beta) $ for each $ \beta \in \Per_F $ does not contain any other point in $ \Per_F $. 
\ssk \\
There exist two disjoint neighborhoods $ B^n_{{\bf w}_n} $ and $ B^n_{{\bf w}'_n} $ of $ w \in \OO_F $ and $ w' \in \OO_F $ respectively for all sufficiently large $ n $. Both $ B^n_{{\bf w}_n} $ and $ B^n_{{\bf w}'_n} $ are forward invariant under $ F^{2^{n+1}} $. We may assume that $ \overline{B^n_{{\bf w}_n}} $ and $ \overline{B^n_{{\bf w}'_n}} $ are disjoint and the minimal distance, $ \dist_{\min}(B^n_{{\bf w}_n}, B^n_{{\bf w}'_n}) \geq \eps_0 > 0 $ for all large enough $ n $. Suppose that both $ w $ and $ w' $ are contained in the same stable manifold, $ W^{s}(w) $ or $ W^{s}(w') $. However, $ \dist_W(w, w') \geq \eps_0 $ for all $ n \in \N $. It contradicts the uniform contraction along strong stable manifold. 
Hence, $ W^{s}(w) \cap \Per_F = \{w\} $ for each $ w \in \Per_F $.
\end{proof}


\end{appendices}

\bsk


\bibliographystyle{alpha}


\end{document}